\documentclass[10pt]{article}

\usepackage[utf8]{inputenc}
\usepackage[english]{babel}
\usepackage{csquotes}

\usepackage{amsfonts, amsmath, amssymb, amsthm}
\usepackage{mathtools}
\usepackage{mathrsfs}
\usepackage{bbm}
\usepackage{breqn}
\usepackage{bm}
\usepackage{cancel}

\usepackage{fullpage}

\usepackage{ragged2e}
\usepackage{enumerate}
\usepackage{enumitem}
\usepackage{multicol}
\usepackage{tasks}

\usepackage{graphicx}
\usepackage{float}
\usepackage{wrapfig}
\usepackage{animate}
\usepackage{caption}
\usepackage{subcaption}
\usepackage{mwe}

\usepackage{tikz}
\usetikzlibrary{arrows.meta}
\usetikzlibrary{arrows}
\usepackage{tikz-cd}
\usepackage{pgfplots}
\pgfplotsset{compat=1.15}

\usetikzlibrary{shapes.geometric, positioning, calc, automata}

\definecolor{white}{rgb}{1,1,1}
\definecolor{xdxdff}{rgb}{0.49019607843137253,0.49019607843137253,1}
\definecolor{ududff}{rgb}{0.30196078431372547,0.30196078431372547,1}
\definecolor{qqffff}{rgb}{0,1,1}
\definecolor{qqffqq}{rgb}{0,1,0}
\definecolor{ffqqqq}{rgb}{1,0,0}
\definecolor{qqqqff}{rgb}{0,0,1}
\definecolor{qqqqcc}{rgb}{0,0,0.8}
\definecolor{ccqqqq}{rgb}{0.8,0,0}
\definecolor{grey}{rgb}{0.4,0.4,0.4}
\definecolor{zzffzz}{rgb}{0.85,1,0.85}
\definecolor{zzffff}{rgb}{0.85,1,1}

\usepackage{hyperref}
\usepackage{cleveref}

\usepackage[backend=biber, url=false, isbn=false, eprint=false, maxbibnames=99, doi=false, maxcitenames=99]{biblatex}
\newtheorem{theorem}{Theorem}[section]

\theoremstyle{definition}
\newtheorem{definition}{Definition}[section]

\theoremstyle{plain}
\newtheorem{prop}[theorem]{Proposition}
\newtheorem{lemma}[theorem]{Lemma}
\newtheorem{corollary}[theorem]{Corollary}

\theoremstyle{remark}

\newcommand{\lno}{\left(}
\newcommand{\rno}{\right)}
\newcommand{\lsq}{\left[}
\newcommand{\rsq}{\right]}
\newcommand{\lcu}{\left\{}
\newcommand{\rcu}{\right\}}

\newcommand{\bbS}{\mathbb{S}}

\newcommand{\lk}{\mathrm{lk}}
\newcommand{\del}{\mathrm{del}}

\newcommand{\medsize}{\fontsize{11pt}{13pt}\selectfont}

\newcommand{\ind}{\operatorname{Ind}}

\newcommand{\caseheading}[1]{%
  \par\medskip\noindent\textbf{#1}\quad\ignorespaces
}

\usepackage{lipsum}
\usepackage{blindtext}

\let\svthefootnote\thefootnote
\newcommand\freefootnote[1]{%
  \let\thefootnote\relax%
  \footnotetext{#1}%
  \let\thefootnote\svthefootnote%
}

\makeatother

\title{\textbf{Independence Complexes of Hexagonal Grid Graphs}}
\author{
    Himanshu Chandrakar\thanks{Indian Institute of Technology Bhilai, \texttt{himanshuc@iitbhilai.ac.in}} 
    \and
    Anurag Singh\thanks{Indian Institute of Technology Bhilai, \texttt{anurags@iitbhilai.ac.in}}
}
\date{}

\begin{document}

\maketitle

\freefootnote{2020 \textit{Mathematics Subject Classification.} 05C69, 55U10, 05E45, 55P15, 57M15}
\freefootnote{\textit{Keywords and phrases.} Independence complex; hexagonal grid graph; homotopy type; fold lemma; link and deletion.
}

\begin{abstract}

The independence complex of a graph is a simplicial complex whose faces correspond to the independent sets of $G$. While independence complexes have been studied extensively for many graph classes, including square grid graphs, relatively little is known about planar hexagonal grid graphs.

In this article, we study the topology of the independence complexes of hexagonal grid graphs $H_{1 \times m \times n}$. For $ m=1, 2, 3$ and $n\geq 1$, we determine their homotopy types. In particular, we show that the independence complex of the hexagonal line tiling $H_{1 \times 1 \times n}$ is homotopy equivalent to a wedge of two $n$-spheres, and $m=2$ and $m=3$ we obtain recursive descriptions that completely determine the spheres appearing in the homotopy type. Our proofs rely on link and deletion operations, the fold lemma, and a detailed analysis of induced subgraphs.



\end{abstract}

\section{Introduction}
In topological combinatorics, independence complexes constitute a central and well-studied class of graph-based simplicial complexes. For a graph $G$, the independence complex $\ind \lno G \rno$ is the simplicial complex with vertex set $V \lno G \rno$, where a subset $\sigma \subseteq V \lno G \rno$ is a face of $\ind \lno G \rno$ if and only if $\sigma$ is an independent set in $G$. Despite this elementary definition, independence complexes often exhibit remarkably rich and subtle topological behaviour. Their significance is further highlighted by their close connection with matching complexes: the matching complex of a graph can be realized as the independence complex of its line graph. Consequently, many developments in the study of matching complexes draw heavily on results and techniques originating from the theory of independence complexes (see, for instance, \cite{BH17,Matsushitamatching}).

Over the years, numerous authors have contributed to the study of independence complexes for a wide variety of graph families. Notable examples include powers of cycles~\cite{ADAMASZEK20121031}, ternary graphs~\cite{jinha_kim_induced_cycles2022,ZhangWu2025Betti}, generalized Mycielskian of complete graphs and categorical products of complete graphs~\cite{Shukla2021Homotopy}, forests and claw-free graphs~\cite{engstrom_clawfree,kawmurahomotopytype}, as well as directed trees and related graph families~\cite{engstrom,KozlovDirectedTrees}. Alongside these results, several powerful methods have been developed to analyze the topology of independence complexes, including the matching tree algorithm, the star cluster method, the fold lemma, and the systematic use of link and deletion operations (see, for instance, \cite{Barmakstarclusters,bousquet2008independence,engstrom,kawmurahomotopytype,Matsushitantimes4,Matsushitantimes6}). Furthermore, in~\cite{jonsson2008simplicial}, Jonsson studies complexes of forests in a matroid-theoretic framework, leading to generalizations of independence complexes of matroids. Additionally, his work serves as a systematic reference for a wide range of simplicial complexes arising from graphs.


Among graph families with an underlying geometric structure, polygonal grid graphs have played a particularly important role. The independence complexes of square grid graphs have been studied extensively, for example by Bousquet-M{\`e}lou et al.~\cite{bousquet2008independence}, Benjamin et al.~\cite{BH17}, and Matsushita et al.~\cite{Matsushitantimes4,Matsushitantimes6}, leading to detailed descriptions of their homotopy types in several cases.

In contrast, comparatively little is known about the independence complexes of planar hexagonal grid graphs. The hexagonal tiling has a fundamentally different local structure from the square grid, with lower vertex degrees and distinct adjacency patterns, which give rise to new combinatorial and topological phenomena. Early work in this direction can be traced back to Thapper~\cite{thapper2008independence}, who studied the independence complexes of square and hexagonal grids in cylindrical settings using the matching tree algorithm. Subsequently, Jonsson~\cite{Jonsson_ind_grid} investigated the occurrence of special homology classes, known as cross cycles, in the independence complexes of periodic grids formed from squares, triangles, and hexagons.

While these works provide valuable insight into independence complexes of periodic and cylindrical grids, the corresponding complexes for \emph{planar} hexagonal grid graphs remain largely unexplored. In this article, we initiate a systematic study of this problem by determining the homotopy type of the independence complex of $H_{1\times m \times n}$ for $m = 1, 2,$ and $3$ (see \Cref{fig: graph of H1n in ind complex}, \Cref{fig: graph of H(1_2_n) in ind complex}, and \Cref{fig: graph of H(1_3_n) in ind complex}, respectively) and $n \geq 1$. The graphs $H_{1 \times m \times n}$ form a natural family of planar hexagonal grids, and the cases $m = 1, 2,$ and $3$ already exhibit distinct homotopy behaviour and much of the complexity present in wider hexagonal grids. As such, they provide a natural starting point for a systematic investigation.

We determine the homotopy types of the independence complexes for each of these cases. In particular, for $m=1$, we obtain a complete and explicit description.

\begin{theorem}[{\Cref{thm: homotopy type of H11n}}]
    For $n \geq 1$, the independence complex of the hexagonal line tiling $H_{1 \times 1 \times n}$ is homotopy equivalent to a wedge of two $n$-dimensional spheres, that is,
    $$\ind \lno H_{1 \times 1 \times n} \rno \simeq \mathbb{S}^n \vee \mathbb{S}^n.$$
\end{theorem}

For the cases $m=2$ and $m=3$, we obtain complete descriptions of the homotopy types in terms of recursive formulas. These formulas determine both the number and the dimensions of the spheres appearing in the homotopy type.

Our computations involve a systematic analysis of several families of intermediate induced subgraphs arising from successive link and deletion operations. A detailed study of these subgraphs leads to a complete classification of the independence complexes of $H_{1 \times 2 \times n}$ and $H_{1 \times 3 \times n}$. We summarize the resulting homotopy types as follows.

\begin{enumerate}
    \item For the case $m=2$, we have the following homotopy equivalence.

    \begin{theorem}[{\Cref{thm: homotopy type of Hn2}}]
        For $n \geq 1$, the independence complex of the hexagonal grid graph $H_{1 \times 2 \times n}$ has the following homotopy type,
        $$\ind \lno H_{1 \times 2 \times n} \rno \simeq \begin{cases}
            \mathbb{S}^2 \vee \mathbb{S}^2, & \text{if } n=1;\\
            \mathbb{S}^4 \vee \mathbb{S}^4 \vee \mathbb{S}^4, & \text{if } n=2;\\
            \mathbb{S}^5 \vee \mathbb{S}^6, & \text{if } n=3;\\
            \Sigma^6 \lno \ind \lno Y_{n-3} \rno \rno \vee \Sigma^4 \lno \ind \lno Y_{n-2} \rno \rno \vee \Sigma^5 \lno \ind \lno H_{n-3}^2 \rno \rno, & \text{if } n \geq 4.
        \end{cases}$$
    \end{theorem}

    Here, $Y_n$ (see \Cref{fig: graph of Yn in ind complex}) is an induced subgraph of $H_{1 \times 2 \times n}$, and the topology of its independence complex is computed in \Cref{thm: homotopy type of Yn in Hn2}.

    \item For the case $m = 3$, we obtain the following result.

    \begin{theorem}[{\Cref{thm: Homotopy type of ind complex of Hn3}}]
        For $n\geq 1$, the independence complex of the hexagonal grid graph $H_{1 \times 3 \times n}$ has the following homotopy type,
        $$\ind \lno H_{1 \times 3 \times n} \rno \simeq \begin{cases}
            \mathbb{S}^3 \vee \mathbb{S}^3, & \text{if } n=1;\\
            \mathbb{S}^4 \vee \mathbb{S}^4 \vee \mathbb{S}^4, & \text{if } n=2;\\
            \mathbb{S}^7, & \text{if } n=3;\\
            \mathbb{S}^6 \vee \mathbb{S}^{10} \vee \mathbb{S}^{10} \vee \mathbb{S}^{10} \vee \mathbb{S}^{10}, & \text{if } n=4;\\
            \Sigma^9 \lno \ind \lno Z_{n-3}^{(3)} \rno \rno \vee \Sigma^4 \lno \ind \lno Z_{n-2}^{(2)} \rno \rno \vee \Sigma^5 \lno \ind \lno Z_{n-2}^{(1)} \rno \rno, & \text{if } n \geq 5.
        \end{cases}$$
    \end{theorem}

    Here, $Z_{n}^{(1)}$ (see \Cref{fig: graph of Zn1 in ind complex}), $Z_{n}^{(2)}$ (see \Cref{fig: graph of Zn2 in ind complex}), and $Z_{n}^{(3)}$ (see \Cref{fig: graph of Zn3 in ind complex}) are induced subgraphs of $H_{1\times 3 \times n}$. Furthermore, the homotopy types of their independence complexes are given in \Cref{thm: Homotopy type of ind complex of Zn1}, \Cref{thm: Homotopy type of ind complex of Zn2}, and \Cref{thm: Homotopy type of ind complex of Zn3}, respectively.

    \end{enumerate}

\caseheading{Organization of the article}

In the following section, we review basic notions from graph theory and simplicial complexes, as well as the results required throughout the article. We also formally define the hexagonal grid graphs. The remainder of the article is organized into three main sections, each devoted to computing the homotopy type of the independence complex of $H_{1\times m \times n}$ for $m = 1, 2,$ and $3$, respectively.

\section{Preliminaries}\label{section: 2 (Preliminaries)}

In this section, we introduce the key concepts and results necessary for discussing the article.

\subsection{Basics of graph theory}

A \textit{\textbf{graph}} $G$ is an ordered pair $\lno V(G), E(G) \rno$, where $V(G)$ is the set of vertices and $$E(G) \subseteq \lcu \{u,v\} \mid u,v \in V(G),\ u \neq v \rcu$$
is the set of edges. We define the \textbf{\textit{(open) neighborhood}} of the vertex $v$ in $G$, denoted by $N_G(v)$ as, $$N_G \lno v \rno = \lcu u \in V(G)\ \middle|\ u\neq v\ \text{and}\ \left\{ u,v \right\} \in E \lno G \rno \rcu.$$ Similarly, the \textbf{\textit{closed neighborhood}} of the vertex $v$ in $G$, denoted by $N_G[v]$ is defined as, $$N_G[v] = N_G(v) \sqcup \lcu v \rcu.$$ When the underlying graph is clear from the context, we omit the subscript and write $N(v)$ and $N[v]$ to denote the open and closed neighborhoods of $v$, respectively.

Let $S$ be a subset of $V(G)$. The \textit{\textbf{induced subgraph}} of $G$ on $S$, denoted by $G\lsq S \rsq$, is the graph whose vertex set is $S$ and whose edge set consists of all edges of $G$ that have both vertices in $S$. We denote the induced
subgraph $G\lsq V\lno G \rno \setminus H \rsq$ by $G \setminus H$, where $H \subset V\lno G \rno$. We now proceed to formally define the hexagonal grid graphs. 

\subsection{Labelling of vertices in the general hexagonal grid graph.}

To understand the topology of the independence complex of hexagonal grid graphs, our computations will largely depend on how their vertices are labeled. For integers $l,m,n \geq 1$, the hexagonal grid graph $H_{l \times m \times n}$ is the finite subgraph of the hexagonal tiling corresponding to a hexagon-shaped region with three pairs of opposite sides of lengths $l$, $m$, and $n$, counted in numbers of hexagons.

For this article, we consider $l=1$ and $m \in \lcu 1, 2, 3 \rcu$ and $n \geq 1$. The $\lno 1 \times m \times n\rno$-hexagonal grid graph, denoted by $H_{1\times m \times n}$, is defined as the graph with vertex set $V\lno H_{1\times m \times n} \rno$ and $E\lno H_{1\times m \times n} \rno$, where
\begin{align*}
    V\lno H_{1\times m \times n} \rno &\simeq \bigsqcup_{i=1}^{m+1} V_i^n;\\
    E\lno H_{1\times m \times n} \rno &\simeq \lno \bigsqcup_{j=1}^{m+1} A_j \rno \sqcup \lno \bigsqcup_{k=1}^{m} B_k \rno .
\end{align*}

Here, for $i = 1, 2, \dots, m+1$, $$V_i^n = \begin{cases}
    \lcu v_1^i, v_2^i, \dots, v_{2n+1}^i\rcu, & i = 1 \text{ or } m+1;\\
    \lcu v_1^i, v_2^i, \dots, v_{2n+1}^i, v_{2n+2}^i\rcu, & \text{otherwise}.
\end{cases}$$

For $j=1,2,\dots, m+1$, 
\begin{align*}
    A_{j} &= \begin{cases}
        \lcu \lcu v_r^j, v_{r+1}^j \rcu\ \middle|\ r \in [2n] \rcu, & j = 1 \text{ or } m+1; \\
        \lcu \lcu v_s^j, v_{s+1}^j \rcu\ \middle|\ s \in [2n+1] \rcu, & \text{otherwise}.
    \end{cases}
\end{align*}

Furthermore, for $m \geq 1$, and $k = 1,2,\dots,m$,
\begin{align*}
    B_{k} &= \begin{cases}
            \lcu \lcu v_{2t-1}^{k}, v_{2t}^{k+1} \rcu\ \middle|\ t \in \lsq n+1 \rsq \rcu, & \text{if } k=1, 2, \dots, m-1;\\
            \lcu \lcu v_{2t-1}^{m}, v_{2t-1}^{m+1} \rcu\ \middle|\ t \in \lsq n+1 \rsq \rcu, & \text{if } k = m.
    \end{cases}
\end{align*}

For example, the graph $H_{1\times 4 \times 5}$ is shown in (see \Cref{fig: graph of H(1_4_6) in ind complex}).

\begin{figure}[h]
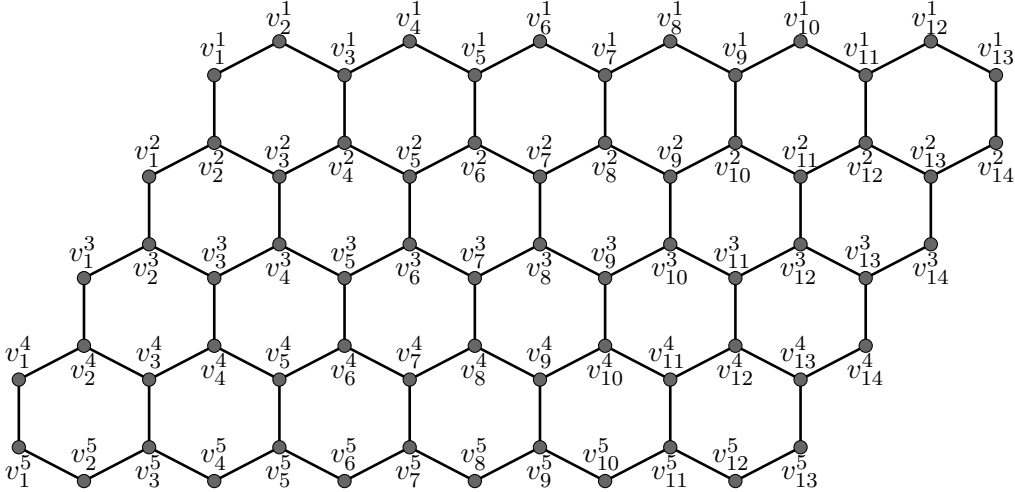

\centering

    \caption{The graph $H_{1 \times 4 \times 6}$.}
    \label{fig: graph of H(1_4_6) in ind complex}
\end{figure}

\subsection{Topological preliminaries and tools for independence complexes}

We follow Hatcher’s \textit{Algebraic Topology}~\cite{hatcher2005algebraic} and Kozlov’s \textit{Combinatorial Algebraic Topology}~\cite{kozlov2008combinatorial} as our primary references for standard definitions and terminology.

A \textbf{\textit{simplicial complex}} $\mathcal{K}$ is a non-empty family of subsets of a finite set $V$, referred to as the vertex set, with the property that it is closed under taking subsets, that is, for every $\tau \in \mathcal{K}$ and $\sigma \subset \tau$, we have $\sigma \in \mathcal{K}$. We call an element $\sigma \in K$ a \textit{face} of $K$. A face consisting of $k+1$ vertices is called a $k$-dimensional face, and the maximum dimension among all faces is called the dimension of the simplicial complex. The maximal faces of a simplicial complex, that is, those faces not properly contained in any other face, are called \emph{facets}. A subset $\mathcal{L} \subseteq \mathcal{K}$ is called a \textit{subcomplex} of $\mathcal{K}$ if $\mathcal{L}$ itself is a simplicial complex and every face of $\mathcal{L}$ is also a face of $\mathcal{K}$. We now discuss some important subcomplexes associated with a given simplicial complex, which will be used extensively throughout this article. The link and deletion of a face in a simplicial complex are defined as follows.

\begin{definition}[{\cite[Definitions~2.12~and~2.13]{kozlov2008combinatorial}}]\label{defn: link and deletion}
    Let $\mathcal{K}$ be a simplicial complex, and $\sigma$ be a face of $\mathcal{K}$. Then,
    \begin{enumerate}[label = \arabic*.]
        \item The \textit{\textbf{link}} of $\sigma$ in $\mathcal{K}$, denoted by $\lk_{\mathcal{K}}\lno \sigma \rno$,
        $$\lk_{\mathcal{K}}\lno \sigma \rno = \lcu \tau \in \mathcal{K}\ \middle|\ \sigma \cap \tau = \emptyset,\ \sigma \cup \tau \in \mathcal{K} \rcu.$$
        
        \item The \textit{\textbf{deletion}} of $\sigma$ in $\mathcal{K}$, denoted by $\del_{\mathcal{K}}\lno \sigma \rno$, is
        $$\del_{\mathcal{K}}\lno \sigma \rno = \lcu \tau \in \mathcal{K}\ \middle|\ \sigma \nsubseteq \tau \rcu.$$
    \end{enumerate}
\end{definition}

Observe that when $K$ is the independence complex of a graph $G$, that is, $K=\ind(G)$, the link and deletion of a vertex $v\in V(G)$ admit particularly simple descriptions. Specifically,
\begin{align*}
    \lk_{\ind(G)}(v) &= \ind\bigl(G\setminus N_G[v]\bigr);\\
    \del_{\ind(G)}(v) &= \ind\bigl(G\setminus\{v\}\bigr).
\end{align*}
These identities follow directly from the definitions of the link and deletion of a vertex in a simplicial complex, as well as those of the independence complex. Henceforth, $\ind\lno G \setminus N_G\lsq v\rsq \rno$ and $\ind\lno G \setminus \lcu v\rcu \rno$ denote the link and deletion of $v$ in $\ind(G)$, respectively. With this notation in place, we state the following lemma, which will be used extensively in what follows.

\begin{lemma}[{\cite[Proposition~3.1]{ADAMASZEK20121031}}]\label{lemma: finding homotopy using link and deletion ind complex}
     Let $v$ be a vertex of $G$. If the inclusion map $$\ind \lno G \setminus N_G \lsq v \rsq \rno \hookrightarrow \ind \lno G \setminus \lcu v \rcu \rno$$ is null-homotopic, then $$\ind\lno G \rno \simeq \ind \lno G \setminus \lcu v \rcu \rno \vee \Sigma \lno \ind \lno G \setminus N_G \lsq v \rsq \rno \rno.$$
\end{lemma}

The key aspect of \Cref{lemma: finding homotopy using link and deletion ind complex} lies in the requirement that the inclusion map be null-homotopic for the stated homotopy equivalence to hold. In this article, we employ two methods to establish the null-homotopy of the relevant inclusion maps.

The first method is based on comparing the dimensions of the link and the deletion of a chosen vertex. More precisely, if the spheres appearing in the homotopy type of the deletion have strictly larger dimensions than those appearing in the homotopy type of the link, then the inclusion map is null-homotopic.

The second method relies on a basic result from homotopy theory, which says that any continuous map $f\colon Y \to X$ that factors through a contractible space is homotopic to a constant map, and hence is null-homotopic. A general statement of this fact appears as a problem in Hatcher’s book (see \cite[Problem~0.10]{hatcher2005algebraic}). Precisely, we state the following lemma, which we will use to establish null-homotopy in such situations.

\begin{prop}\label{prop: null-homotopy through factor}
    Let $j\colon Y\hookrightarrow Z$ and $i\colon Z\hookrightarrow X$ be inclusion maps. If $Z$ is contractible, then the composite inclusion $i\circ j\colon Y\hookrightarrow X$ is null-homotopic.
\end{prop}

We say that a vertex $w$ is \textit{folded} using a vertex $v$ in a graph $G$ if $N(v) \subseteq N(w)$. The following lemma, commonly known as the \textit{fold lemma} for independence complexes, asserts that under this condition, the removal of a vertex from the graph preserves the homotopy type of the associated independence complex.

\begin{lemma}(\cite[{Lemma~2.4}]{engstrom_clawfree})\label{foldlemma}
    For two distinct vertices $v$ and $w$ of a graph $G$ with $N(v) \subseteq N(w)$, $\ind(G)$ collapses onto $\ind\left(G \setminus \left\{w\right\}\right)$, that is, $$\ind(G) \simeq \ind(G \setminus w).$$
\end{lemma}

The following lemma describes the independence complex of a graph that can be expressed as the disjoint union of two graphs.

\begin{lemma}[{\cite[Remark~2.2]{Barmakstarclusters}}]
    If a graph $G$ is the disjoint union of two graphs $H_1$ and $H_2$, then the independence complex of $G$ is homotopic to the join of the independence complexes of $H_1$ and $H_2$, that is, $$\ind(G) \simeq \ind(H_1) * \ind(H_2).$$ In particular, if $H_1$ consists of a single vertex, then $\ind(G)$ is a cone over $\ind(H_2)$, and hence contractible. Similarly, if $H_1$ is the path graph on two vertices, that is, $P_2$, then $\ind(G) \simeq \Sigma \ind(H_2)$.
\end{lemma}

Along with these results, we also need the homotopy type of the independence complexes of paths and cycles, which will be used extensively throughout this article.

\begin{theorem}[{\cite[Proposition~4.5(3)]{KozlovDirectedTrees}}]\label{thm: homotopy type of ind complex of path graph}
    The independence complex of path graphs, $P_n$, has the following homotopy type,

    $$\ind \lno P_n \rno = \begin{cases}
        \mathbb{S}^{k-1},& n=3k;\\
        *,& n=3k+1;\\
        \mathbb{S}^{k},& n=3k+2.
    \end{cases}$$
\end{theorem}

\begin{theorem}[{\cite[Proposition~5.2]{KozlovDirectedTrees}}]\label{thm: homotopy type of ind complex of cycle graph}
    The independence complex of cycle graphs, $C_n$, has the following homotopy type,

    $$\ind \lno C_n \rno = \begin{cases}
        \mathbb{S}^{k-1} \vee \mathbb{S}^{k-1},& n=3k;\\
        \mathbb{S}^{k-1},& n=3k+1;\\
        \mathbb{S}^{k},& n=3k+2.
    \end{cases}$$
\end{theorem}

\section[Independence Complex of $H_{1\times 1 \times n}$]{Independence Complex of $\texorpdfstring{\bm{H_{1\times 1 \times n}}}{H_{1\times 1 \times n}}$}

In this section, we determine the homotopy type of the independence complex of the $\lno {1\times 1 \times n} \rno$-hexagonal grid graph (see \Cref{fig: graph of H1n in ind complex}), $H_{1\times 1 \times n}$ (or, the \textit{hexagonal line tiling}), that is, $\operatorname{Ind}\lno H_{1 \times 1 \times n}\rno$, for $n \geq 1$.
\begin{figure}[H]
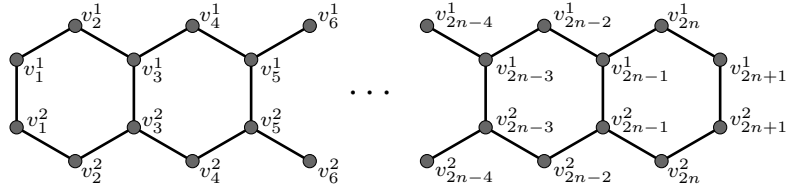

\centering

    \caption{The graph $H_{1 \times 1 \times n}$.}
    \label{fig: graph of H1n in ind complex}
\end{figure}

At first, we state the main result of this section.

\begin{theorem}\label{thm: homotopy type of H11n}
    For $n \geq 1$, the independence complex of the hexagonal line tiling, $H_{1\times 1 \times n}$ is homotopy equivalent to a wedge of two $n$-dimensional spheres, that is,
    $$\ind \lno H_{1\times 1 \times n} \rno \simeq \mathbb{S}^n \vee \mathbb{S}^n.$$
\end{theorem}

We define the following induced subgraphs of $H_{1\times 1 \times n}$ that will be useful in computing the homotopy type on $\ind \lno H_{1\times 1 \times n} \rno$ ,

\begin{enumerate}
    \item $X_n^{(1)} = H_{1\times 1 \times n} \setminus \lcu v_{2n+1}^{1} \rcu$ (see \Cref{fig: graph of Xn1 - ind complex}).
    \item $X_n^{(2)} = H_{1\times 1 \times n} \setminus N \lsq v_{2n+1}^{1} \rsq$ (see \Cref{fig: graph of Xn2 - ind complex}).
\end{enumerate}
\begin{figure}[H]
\centering
    \begin{minipage}{0.45\textwidth}
        \centering

        \caption{The graph $X_n^{(1)}$}
        \label{fig: graph of Xn1 - ind complex}
    \end{minipage}\hfill
    \begin{minipage}{0.45\textwidth}
        \centering
        %
        \caption{The graph $X_{n}^{(2)}$.}
        \label{fig: graph of Xn2 - ind complex}
    \end{minipage}
\end{figure}

The idea here is to look at the link and deletion of the vertex $v_{2n+1}^1$ in $\ind \lno H_{1 \times 1 \times n} \rno$. These are precisely the $\ind \lno H_{1 \times 1 \times n} \setminus N \lsq v_{2n+1}^1 \rsq\rno$ and $\ind \lno H_{1 \times 1 \times n} \setminus \lcu v_{2n+1}^1 \rcu \rno$.

\begin{lemma}\label{lemma: homotopy of Xn-1}
    For $n \geq 1$, the independence complex of $X_n^{(1)}$, $\ind\lno X_n^{(1)} \rno$ has the following homotopy type, $$\ind\lno X_n^{(1)} \rno \simeq \Sigma \lno \ind \lno X_{n-1}^{(2)} \rno \rno.$$
\end{lemma}

\begin{proof}
    Observe that in $X_n^{(1)}$ (see \Cref{fig: graph of Xn1 - ind complex}), we fold $v_{2n-2}^1$ and $v_{2n-1}^2$ using $v_{2n}^1$ (see \Cref{fig: foldings of Xn1}).

    \begin{figure}[H]
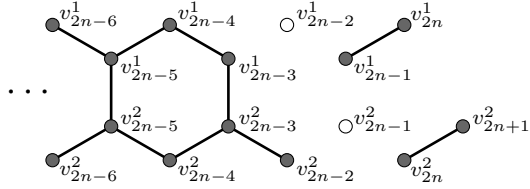

        \centering

        \caption{Foldings in the graph $X_n^{(1)}$}
        \label{fig: foldings of Xn1}
    \end{figure}
    
    By applying the fold Lemma (see \Cref{foldlemma}), we get, $$\ind \lno X_n^{(1)} \rno \simeq \ind \lno X_{n-1}^{(2)} \sqcup P_2 \sqcup P_2\rno
    \simeq \Sigma^2 \lno \ind \lno X_{n-1}^{(2)} \rno \rno.$$

    This completes the proof \Cref{lemma: homotopy of Xn-1}.
\end{proof}

\begin{lemma}\label{lemma: homotopy of Xn2}
    For $n \geq 1$, the independence complex of $X_n^{(2)}$, $\ind\lno X_n^{(2)} \rno$ has the following homotopy type,
    $$\ind\lno X_n^{(2)} \rno \simeq \mathbb{S}^{n-1}.$$
\end{lemma}

\begin{proof}
    We prove this by induction on $n$. The result is immediate for $n=1$, since $X_1^{(2)}$ is precisely the path graph $P_3$, and we know that $\ind \lno P_3 \rno \simeq \bbS^0$ (using \Cref{thm: homotopy type of ind complex of path graph}). Thus, we have $$\ind \lno X_1^2 \rno \simeq \bbS^0.$$
    
    Assuming that the given result is true for all $n \leq t$, we aim to prove it for $n=t+1$. Observe that in $X_{n+1}^2$, we fold $v_{2t+1}^1$ and $v_{2t}^2$ using $v_{2t+2}^2$. Following the folding of $v_{2t+1}^1$, we fold $v_{2t-2}^1$ and $v_{2t-1}^2$ using $v_{2t}^1$ (see \Cref{fig: foldings of Xn2}).

    \begin{figure}[H]
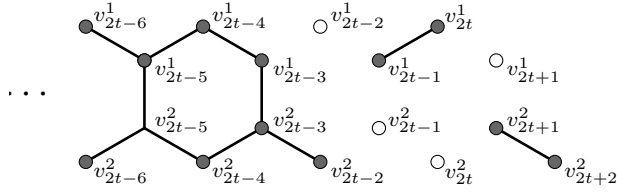

    \centering

        \caption{Foldings in the graph $X_{t+1}^2$.}
        \label{fig: foldings of Xn2}
    \end{figure}
    
    Therefore, using the fold Lemma (see \Cref{foldlemma}), we have $$\ind \lno X_{t+1}^2 \rno \simeq \ind \lno X_{t-1}^2 \sqcup P_2 \sqcup P_2 \rno \simeq \Sigma^2 \lno \ind\lno X_{t-1}^2 \rno \rno.$$

    Thus, using the induction hypothesis, the result holds for $n=t+1$, that is, $$\ind \lno X_{t+1}^2 \rno \simeq \bbS^t .$$

    This completes the proof \Cref{lemma: homotopy of Xn2}.    
\end{proof}

From the two results above, we obtain the following corollary.

\begin{corollary}\label{coro: homotopy type of Xn-2}
    For $n\geq 1$, $\ind \lno X_n^1 \rno \simeq \bbS^n$.
\end{corollary}

With these results in hand, we now proceed to prove \Cref{thm: homotopy type of H11n}.

\begin{proof}[Proof of \Cref{thm: homotopy type of H11n}]
    Observe that it suffices to prove that the inclusion map $$\ind \lno X_n^{(2)} \rno \hookrightarrow \ind \lno X_n^{(1)} \rno$$ is null-homotopic, since the result then follows from \Cref{lemma: finding homotopy using link and deletion ind complex}, \Cref{lemma: homotopy of Xn-1}, and \Cref{lemma: homotopy of Xn2}.

    Using \Cref{lemma: homotopy of Xn2} and \Cref{coro: homotopy type of Xn-2}, it is clear that $\ind \lno X_n^{(2)} \rno$ is homotopy equivalent to a sphere with dimension one less than that of $\ind \lno X_n^{(1)} \rno$. Thus, the above inclusion map is null-homotopic. Therefore, using \Cref{lemma: finding homotopy using link and deletion ind complex}, we have
    \begin{align*}
        \ind \lno H_{1\times 1 \times n} \rno &\simeq \ind \lno X_n^{(1)} \rno \vee \Sigma \lno \ind \lno X_n^{(2)} \rno \rno\\
        &\simeq \bbS^n \vee \Sigma\lno \bbS^{n-1} \rno \\
        &\simeq \bbS^n \vee \bbS^n.
    \end{align*}

    This completes the proof \Cref{thm: homotopy type of H11n}.
\end{proof}

\section[Independence Complex of $H_{1\times 2 \times n}$]{Independence Complex of $\texorpdfstring{\bm{H_{1\times 2 \times n}}}{H_{1\times 2 \times n}}$}

In this section, we determine the homotopy type of the independence complex of the $\lno {1\times 2 \times n} \rno$-hexagonal grid graph (see \Cref{fig: graph of H(1_2_n) in ind complex}), $H_{1\times 2 \times n}$ (or, the \textit{double hexagonal line tiling}), that is, $\operatorname{Ind}\lno H_{1 \times 2 \times n}\rno$, for $n \geq 1$. For simplicity, we write $H_n^2 = H_{1 \times 2 \times n}$.

\begin{figure}[H]
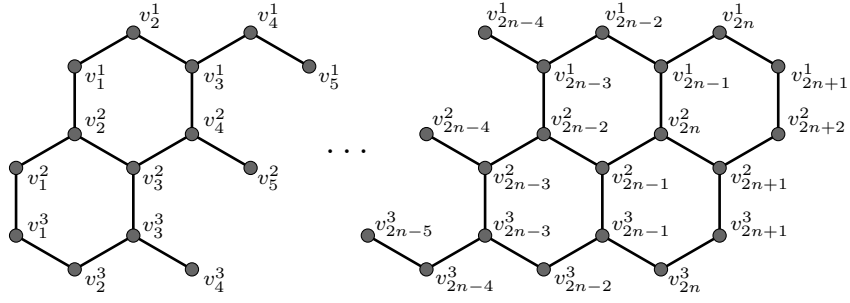

\centering

    \caption{The graph $H_{1 \times 2 \times n}$.}
    \label{fig: graph of H(1_2_n) in ind complex}
\end{figure}

We define the induced subgraph $Y_n$ (see \Cref{fig: graph of Yn in ind complex}) of $H_n^2$ that will be useful in determining the homotopy type on $\ind \lno H_{1 \times 2 \times n} \rno$ as follows: $$Y_n = H_n^2 \setminus \lcu v_{2n}^1, v_{2n+1}^1, v_{2n+2}^2 \rcu.$$
\begin{figure}[H]
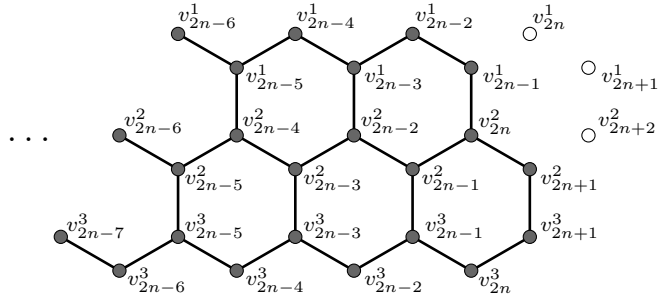

    \centering
        %
        \caption{The graph $Y_n$.}
        \label{fig: graph of Yn in ind complex}
    \end{figure}

\subsection{Independence complex of $Y_n$}

Here, we will determine the homotopy type of the independence complex of $Y_n$, denoted $\ind \lno Y_n \rno$. For this, we first discuss the homotopy type of the link and the deletion of the vertex $v_{2n-1}^2$ in $\ind \lno Y_n \rno$, that is, $\ind \lno Y_n \setminus N \lsq v_{2n-1}^2 \rsq \rno$ and $\ind \lno Y_n \setminus \lcu v_{2n-1}^2 \rcu \rno$, respectively.

\begin{lemma}\label{lemma: link of v(2t-1)2 in Yn}
    For $n \geq 3$, $\ind \lno Y_n \setminus N_{Y_n} \lsq v_{2n-1}^2 \rsq \rno \simeq \Sigma^4 \lno \ind \lno Y_{n-3} \rno \rno$.
\end{lemma}

\begin{figure}[H]
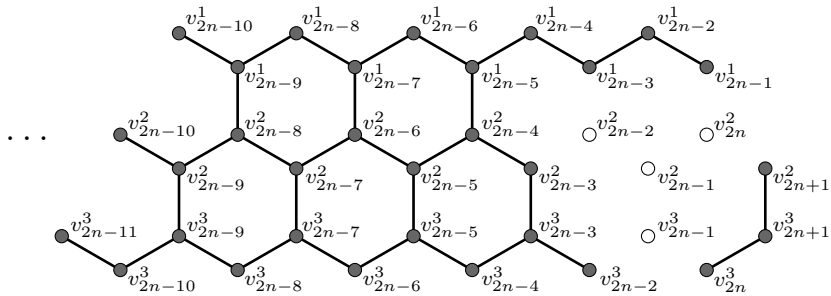

\centering

    \caption{The graph $Y_n \setminus N_{Y_n} \lsq v_{2n-1}^2 \rsq$.}
    \label{fig: graph of Yn - N[v(2n-1)2] in ind complex}
\end{figure}

\begin{proof}
    Observe that in $Y_n \setminus N_{Y_n} \lsq v_{2n-1}^2 \rsq$ (see \Cref{fig: graph of Yn - N[v(2n-1)2] in ind complex}), we fold $v_{2n}^3$ using $v_{2n+1}^2$. Similarly, we fold $v_{2n-3}^1$ using $v_{2n-1}^1$. Once $v_{2n-3}^1$ is folded, we fold $v_{2n-6}^1$ and $v_{2n-4}^2$ using $v_{2n-4}^1$. On the other hand, following the folding of $v_{2n}^3$, we fold $v_{2n-3}^2$ and $v_{2n-4}^3$ using $v_{2n-2}^3$ (see \Cref{fig: foldings in Yn - N[v(2n-1)2]}).

    \begin{figure}[H]
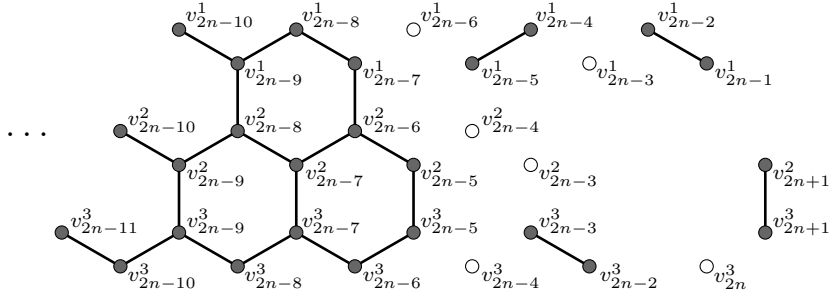

    \centering

        \caption{Folding in the graph $Y_n \setminus N_{Y_n} \lsq v_{2n-1}^2 \rsq$.}
        \label{fig: foldings in Yn - N[v(2n-1)2]}
    \end{figure}
    
    Therefore, using the fold Lemma (\Cref{foldlemma}), we have 
    \begin{align*}
        \ind \lno Y_n \setminus N_{Y_n} \lsq v_{2n-1}^2 \rsq \rno &\simeq \ind \lno Y_{n-3} \sqcup P_2 \sqcup P_2\sqcup P_2\sqcup P_2\rno\\
        &\simeq \Sigma^4 \lno \ind \lno Y_{n-3} \rno \rno.
    \end{align*}

    This completes the proof of \Cref{lemma: link of v(2t-1)2 in Yn}.
\end{proof}

\begin{lemma}\label{lemma: deletion of v(2t-1)2 in Yn}
    For $n \geq 5$, $$\ind \lno Y_n \setminus \lcu v_{2n-1}^2 \rcu \rno \simeq \Sigma^4 \lno \ind \lno H_{n-3}^2 \rno \rno \vee \Sigma^6 \lno \ind \lno H_{n-4}^2 \rno \rno.$$
\end{lemma}

\begin{proof}
    Let $Y_n^{(1)} = Y_n \setminus \lcu v_{2n-1}^2 \rcu$ (see \Cref{fig: graph of Yn - v(2n-1)2 in ind complex}).

    \begin{figure}[H]
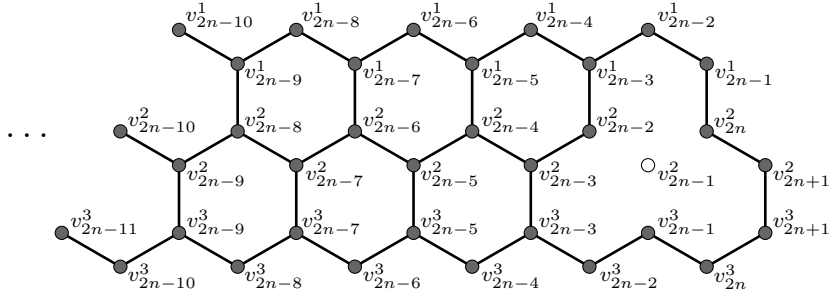

    \centering

        \caption{The graph $Y_n^1 = Y_n \setminus \lcu v_{2n-1}^2 \rcu$.}
        \label{fig: graph of Yn - v(2n-1)2 in ind complex}
    \end{figure}
    
    We first discuss the homotopy type of

    \begin{enumerate}
        \item $\ind \lno Y_n^{(1)} \setminus  N_{Y_n^{(1)}} \lsq v_{2n-1}^3 \rsq \rno$ and,
        \item $\ind \lno Y_n^{(1)} \setminus \lcu v_{2n-1}^3 \rcu \rno$.
    \end{enumerate}
    
    Later, we show that the inclusion map $$\ind \lno Y_n^{(1)} \setminus  N_{Y_n^{(1)}} \lsq v_{2n-1}^3 \rsq \rno \hookrightarrow \ind \lno Y_n^{(1)} \setminus \lcu v_{2n-1}^3 \rcu \rno,$$ is null-homotopic.

    \caseheading{Homotopy type of $\bm{{\ind} \lno Y_{n}^{(1)} \setminus  N_{Y_{n}^{(1)}} \lsq v_{2n-1}^3 \rsq \rno}$:}

    \begin{figure}[H]
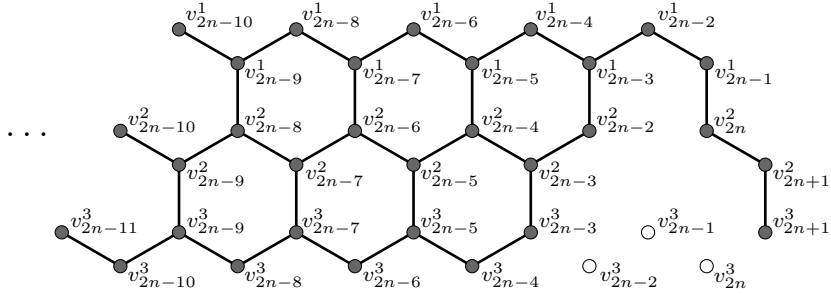

    \centering

        \caption{The graph $Y_{n}^{(2)} =  Y_n^{(1)} \setminus  N_{Y_n^{(1)}} \lsq v_{2n-1}^3 \rsq$.}
        \label{fig: graph of Yn1 - v(2n-1)3 in ind complex}
    \end{figure}
    
    Let $Y_{n}^{(2)} =  Y_n^{(1)} \setminus  N_{Y_n^{(1)}} \lsq v_{2n-1}^3 \rsq$ (see \Cref{fig: graph of Yn1 - v(2n-1)3 in ind complex}). Observe that in $Y_{n}^{(2)}$, we fold $v_{2n}^2$ using $v_{2n+1}^3$. Once $v_{2n}^2$ is folded, we fold $v_{2n-3}^1$ using $v_{2n-1}^1$. After the folding of $v_{2n-3}^1$, we fold $v_{2n-6}^1$ and $v_{2n-4}^2$ using $v_{2n-4}^1$, and $v_{2n-3}^3$ using $v_{2n-2}^2$. Finally, following the folding of $v_{2n-3}^3$, we fold $v_{2n-5}^2$ and $v_{2n-6}^3$ using $v_{2n-4}^3$ (see \Cref{fig: foldings in Yn2}).

    \begin{figure}[H]
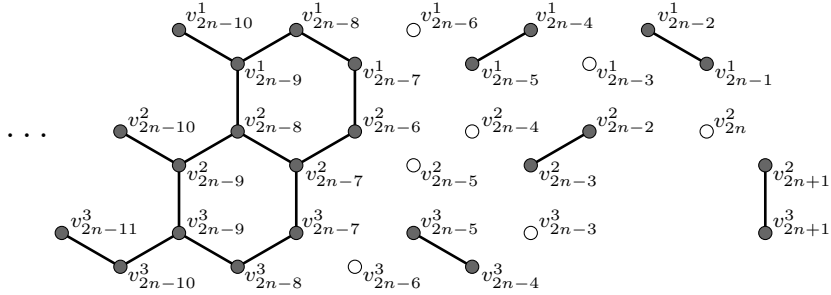

    \centering

        \caption{Foldings in the graph $Y_{n}^{(2)}$.}
        \label{fig: foldings in Yn2}
    \end{figure}

    Therefore, using the fold Lemma (\Cref{foldlemma}), we have 
    \begin{align*}
        \ind \lno Y_n^{(2)} \rno &\simeq \ind \lno H_{n-4}^2 \sqcup P_2 \sqcup P_2 \sqcup P_2\sqcup P_2\sqcup P_2\rno\\
        &\simeq \Sigma^5 \lno \ind \lno H_{n-4}^2 \rno \rno.
    \end{align*}

    Hence, 
    \begin{align}\label{eqn: link of c(2t-1) in Yn(1)}
        \ind \lno Y_n^{(1)} \setminus  N_{Y_n^{(1)}} \lsq v_{2n-1}^3 \rsq \rno \simeq \Sigma^5 \lno \ind \lno H_{n-4}^2 \rno \rno.
    \end{align}

    \caseheading{Homotopy type of $\bm{{\ind} \lno Y_{n}^{(1)} \setminus \lcu v_{2n-1}^3 \rcu \rno}$:}

    \begin{figure}[H]
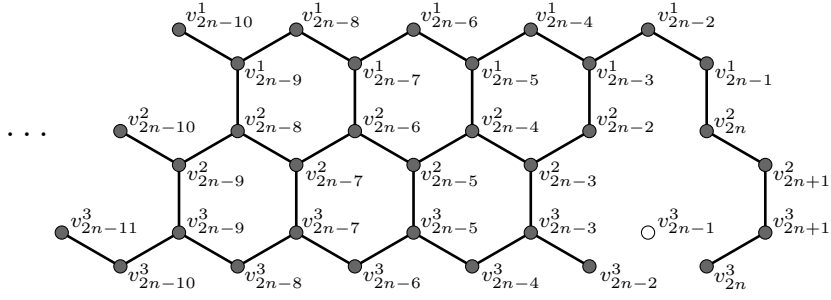

    \centering

        \caption{The graph $Y_{n}^{(3)} =  Y_n^{(1)} \setminus \lcu v_{2n-1}^3 \rcu$.}
        \label{fig: graph of Yn3 in ind complex}
    \end{figure}

    Let $Y_n^{(3)} = Y_n^{(1)} \setminus \lcu v_{2n-1}^3 \rcu$ (see \Cref{fig: graph of Yn3 in ind complex}). Observe that in $Y_n^{(3)}$, we fold $v_{2n+1}^2$ using $v_{2n}^3$. Similarly, we fold $v_{2n-3}^2$ and $v_{2n-4}^3$ using $v_{2n-2}^3$. Following the folding of $v_{2n-3}^2$, we fold $v_{2n-4}^1$ and $v_{2n-2}^1$ using $v_{2n-2}^2$ (see \Cref{fig: foldings in Yn3}).

    \begin{figure}[H]
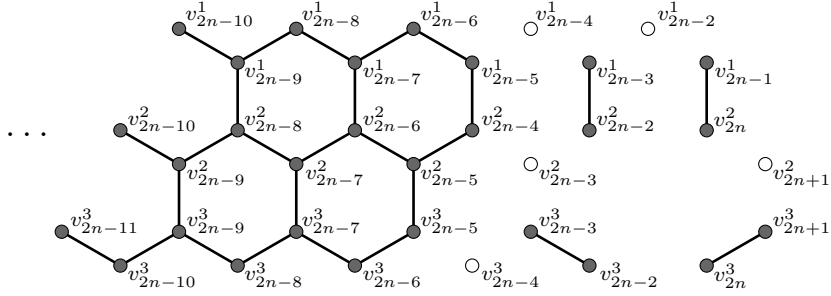

    \centering
        %
        \caption{Foldings in the graph $Y_{n}^{(3)}$.}
        \label{fig: foldings in Yn3}
    \end{figure}

    Therefore, using the fold Lemma (\Cref{foldlemma}), we have
    \begin{align*}
        \ind \lno Y_n^{(3)} \rno &\simeq \ind \lno H_{n-3}^2 \sqcup P_2 \sqcup P_2 \sqcup P_2\sqcup P_2\rno\\
        &\simeq \Sigma^4 \lno H_{n-3}^2 \rno.
    \end{align*}

     Hence,
     \begin{align}\label{eqn: deletion of c(2t-1) in Yn(1)}
         \ind \lno Y_n^{(1)} \setminus \lcu v_{2n-1}^3 \rcu \rno \simeq \Sigma^4 \lno H_{n-3}^2 \rno.
     \end{align}

    \caseheading{The inclusion map $\bm{{\ind} \lno Y_{n}^{(1)} \setminus  N_{Y_{n}^{(1)}} \lsq v_{2n-1}^3 \rsq \rno} \hookrightarrow \bm{{\ind} \lno Y_{n}^{(1)} \setminus \lcu v_{2n-1}^3 \rcu \rno}$  is null-homotopic:}

    To prove this, we consider the induced subgraph $Y_n^{(4)} = Y_n^{(1)} \setminus \lcu v_{2n-1}^3,\ v_{2n}^3 \rcu$, that is, $Y_n^{(4)}$ is precisely $Y_n^{(2)} \cup \lcu v_{2n-2}^3 \rcu$ (see \Cref{fig: graph of Yn4 in ind complex}).

    \begin{figure}[H]
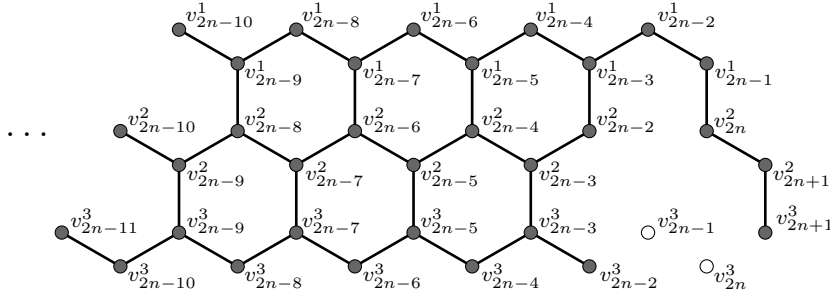

    \centering

        \caption{The graph $Y_n^{(4)} = Y_n^{(1)} \setminus \lcu v_{2n-1}^3,\ v_{2n}^3 \rcu$.}
        \label{fig: graph of Yn4 in ind complex}
    \end{figure}
    
    We then have the following inclusion map, 
    $$\ind \lno Y_n^{(1)} \setminus  N_{Y_n^{(1)}} \lsq v_{2n-1}^3 \rsq \rno \hookrightarrow \ind \lno Y_n^{(4)} \rno \hookrightarrow \ind \lno Y_n^{(1)} \setminus \lcu v_{2n-1}^3 \rcu \rno.$$

    Observe that in $Y_n^{(4)}$, we fold $v_{2n-3}^2$ and $v_{2n-4}^3$ using $v_{2n-2}^3$. Once $v_{2n-3}^2$ is folded, we fold $v_{2n-4}^1$ and $v_{2n-2}^1$ using $v_{2n-2}^2$. At last, following the folding of $v_{2n-2}^1$, we fold $v_{2n+1}^2$ using $v_{2n-1}^1$. This whole procedure leaves $v_{2n+1}^3$ isolated (see \Cref{fig: foldings in Yn4}). 
    
    \begin{figure}[H]
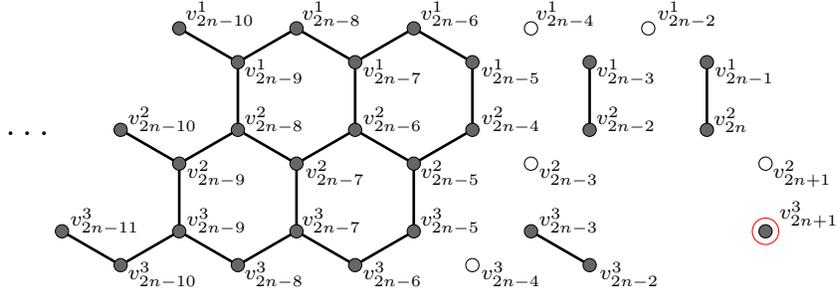

    \centering

        \caption{Foldings in the graph $Y_{n}^{(4)}$.}
        \label{fig: foldings in Yn4}
    \end{figure}
    
    Now, applying the fold Lemma (\Cref{foldlemma}), we get
    \begin{align*}
        \ind \lno Y_n^{(4)} \rno &\simeq \ind \lno H_{n-3}^2 \sqcup P_2 \sqcup P_2 \sqcup P_2\sqcup \lcu v_{2n+1}^3 \rcu \rno\\
        &\simeq \lno \Sigma^3 \lno \ind \lno H_{n-3}^2 \rno \rno \rno * \lcu v_{2n+1}^3 \rcu\\
        &\simeq pt.
    \end{align*}

    In other words, $\ind \lno Y_n^{(4)} \rno$ is contractible. Together with \Cref{prop: null-homotopy through factor}, this implies that the inclusion map,
    $$\ind \lno Y_n^{(1)} \setminus  N_{Y_n^{(1)}} \lsq v_{2n-1}^3 \rsq \rno \hookrightarrow \ind \lno Y_n^{(4)} \rno,$$
    is null-homotopic. Which in turn implies that, $$\ind \lno Y_n^{(1)} \setminus  N_{Y_n^{(1)}} \lsq v_{2n-1}^3 \rsq \rno \hookrightarrow \ind \lno Y_n^{(1)} \setminus \lcu v_{2n-1}^3\rcu \rno,$$
    is null-homotopic. Consequently, using \Cref{lemma: finding homotopy using link and deletion ind complex}, \Cref{eqn: link of c(2t-1) in Yn(1)} and \Cref{eqn: deletion of c(2t-1) in Yn(1)}, we have
    \begin{align*}
        \ind\lno Y_n^{(1)} \rno &\simeq \ind \lno Y_n^{{(1)}} \setminus \lcu v_{2n-1}^3 \rcu \rno \vee \Sigma \lno \ind \lno Y_n^{(1)} \setminus  N_{Y_n^{(1)}}\lsq v_{2n-1}^3 \rsq \rno \rno\\
        &\simeq \Sigma^4 \lno \ind \lno H_{n-3}^2 \rno \rno \vee \Sigma\lno \Sigma^5\lno \ind \lno H_{n-4}^2 \rno \rno\rno\\
        &\simeq \Sigma^4 \lno \ind \lno H_{n-3}^2 \rno \rno \vee \Sigma^6 \lno \ind \lno H_{n-4}^2 \rno \rno
    \end{align*}

    Hence, $$\ind \lno Y_n \setminus \lcu v_{2n-1}^2 \rcu \rno \simeq \Sigma^4 \lno \ind \lno H_{n-3}^2 \rno \rno \vee \Sigma^6 \lno \ind \lno H_{n-4}^2 \rno \rno.$$

    This completes the proof of \Cref{{lemma: deletion of v(2t-1)2 in Yn}}.
\end{proof}

With these two results in hand, we now state and prove the following theorem, which determines the homotopy type of $\ind \lno Y_n \rno$. Recall, $$Y_n = H_n^2 \setminus \lno v_{2n}^1, v_{2n+1}^1, v_{2n+2}^2 \rno.$$

\begin{theorem}\label{thm: homotopy type of Yn in Hn2}
    For $n\geq 1$, 

        $$\ind \lno Y_n \rno \simeq \begin{cases}
            \mathbb{S}^1, & \text{if } n=1;\\
            \mathbb{S}^3 \vee \mathbb{S}^3 \vee \mathbb{S}^3, & \text{if } n=2;\\
            \mathbb{S}^5 \vee \mathbb{S}^5, & \text{if } n=3;\\
            \mathbb{S}^6 \vee \mathbb{S}^6 \vee \mathbb{S}^6, & \text{if } n=4;\\
            \Sigma^4 \lno \ind \lno H_{n-3}^2 \rno \rno \vee \Sigma^6 \lno \ind \lno H_{n-4}^2 \rno \rno \vee \Sigma^5 \lno \ind \lno Y_{n-3} \rno \rno, & \text{if } n \geq 5.
        \end{cases}$$
\end{theorem}

\begin{proof}
    For the case $n = 1$, observe that $Y_1$ is isomorphic to the graph $X_{2}^{(2)}$ (see \Cref{fig: graph of Xn2 - ind complex}). Hence, the result follows directly from \Cref{lemma: homotopy of Xn2}. We skip the details of the proof for the cases when $n = 2,3$ and $4$, since the proof proceeds in the same way as in \Cref{lemma: link of v(2t-1)2 in Yn} and \Cref{lemma: deletion of v(2t-1)2 in Yn}. That is, we analyze the link and deletion of the vertex $v_{2n-1}$ and perform a sequence of foldings analogous to those used earlier in each case. The only difference is that the resulting intermediate subgraphs are slightly different, although each of them is simple enough that the fold lemma can be applied again to determine their homotopy types. Moreover, showing that the inclusion map is null-homotopic becomes easier for these cases, since the link has homotopy type with spheres of dimension strictly smaller than that of the deletion.
    
    For the case when $n \geq 5$, observe that it suffices to prove that the inclusion map
    $$\ind \lno Y_n \setminus  N_{Y_n} \lsq v_{2n-1}^2 \rsq \rno \hookrightarrow \ind \lno Y_n \setminus \lcu v_{2n-1}^2 \rcu \rno$$ is null-homotopic, since the result then follows from \Cref{lemma: finding homotopy using link and deletion ind complex}, and \Cref{lemma: link of v(2t-1)2 in Yn}, \Cref{lemma: deletion of v(2t-1)2 in Yn}.

    For this, we consider the induced subgraph $Y_n^{(5)} = Y_n \setminus \lcu v_{2n-1}^2, v_{2n}^2, v_{2n-1}^3 \rcu$ (see \Cref{fig: graph of Yn5 in ind complex}).
    
    \begin{figure}[H]
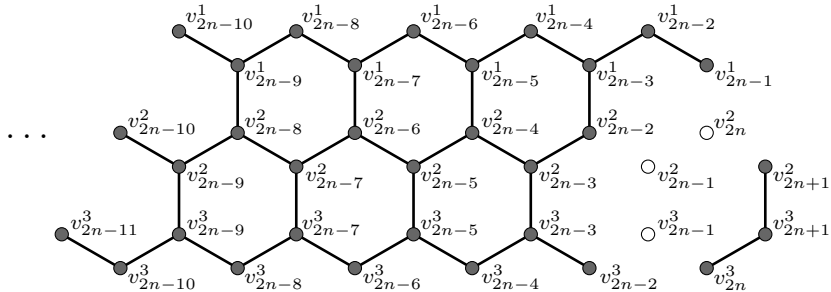

    \centering

        \caption{The graph $Y_n^{(5)} = Y_n \setminus \lcu v_{2n-1}^2, v_{2n}^2, v_{2n-1}^3 \rcu$.}
        \label{fig: graph of Yn5 in ind complex}
    \end{figure}
    
    We then have the following inclusion map, 
    $$\ind \lno Y_n \setminus  N_{Y_n} \lsq v_{2n-1}^2 \rsq \rno \hookrightarrow \ind \lno Y_n^{(5)} \rno \hookrightarrow \ind \lno Y_n\setminus \lcu v_{2n-1}^2 \rcu \rno.$$

    At first, note that we fold $v_{2n}^3$ using $v_{2n+1}^2$ in $Y_n^{(5)}$. On the other hand, we fold $v_{2n-3}^2$ and $v_{2n-4}^3$ using $v_{2n-2}^3$. Once $v_{2n-3}^2$ is folded, we fold $v_{2n-4}^1$ and $v_{2n-2}^1$ using $v_{2n-2}^2$. Observe that this whole procedure leaves $v_{2n-1}^1$ isolated (see \Cref{fig: foldings in Yn5}).

    \begin{figure}[H]
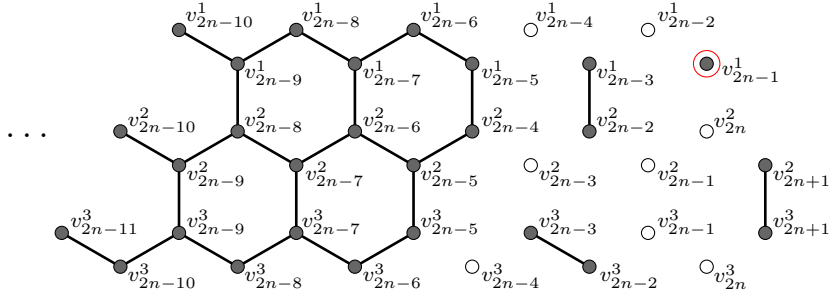

    \centering

        \caption{Foldings in the graph $Y_n^{(5)}$.}
        \label{fig: foldings in Yn5}
    \end{figure}
    
    Thus, using the fold Lemma (\Cref{foldlemma}), we have
    \begin{align*}
        \ind \lno Y_n^{(5)} \rno &\simeq \ind \lno H_{n-3}^2 \sqcup P_2 \sqcup P_2 \sqcup P_2\sqcup \lcu v_{2n-1}^1 \rcu \rno\\
        &\simeq \lno \Sigma^3 \lno \ind \lno H_{n-3}^2 \rno \rno \rno * \lcu v_{2n-1}^1 \rcu\\
        &\simeq pt.
    \end{align*}

    In other words, $\ind \lno Y_n^{(5)} \rno$ is contractible. Together with \Cref{prop: null-homotopy through factor}, this implies that the inclusion map,
    $$\ind \lno Y_n \setminus  N_{Y_n} \lsq v_{2n-1}^2 \rsq \rno \hookrightarrow \ind \lno Y_n^{(5)} \rno,$$
    is null-homotopic. Which in turn implies that, $$\ind \lno Y_n \setminus  N_{Y_n} \lsq v_{2n-1}^2 \rsq \rno \hookrightarrow \ind \lno Y_n \setminus \lcu v_{2n-1}^2\rcu \rno,$$
    is null-homotopic. This completes the proof of \Cref{thm: homotopy type of Yn in Hn2}.
\end{proof}

\subsection{Independence complex of $H_{n}^{2}$}

We now discuss the homotopy type of the independence complex of $H_{n}^{2}$, denoted $\ind \lno H_n^2 \rno$. As before, we first discuss the homotopy type of the link and deletion of the vertex $v_{2n}^2$ in $\ind \lno H_n^2 \rno$, that is, $\ind \lno H_n^2 \setminus N_{H_n^2} \lsq v_{2n}^2 \rsq\rno$ and $\ind \lno H_n^2 \setminus \lcu v_{2n}^2 \rcu\rno$, respectively.

\begin{lemma}\label{lemma: link of b(2t) in Hn2}
    For $n\geq 4$, $\ind \lno H_n^2 \setminus  N_{H_n^2} \lsq v_{2n}^2 \rsq \rno \simeq \Sigma^4 \lno \ind \lno H_{n-3}^2 \rno \rno$.
\end{lemma}

\begin{figure}[H]
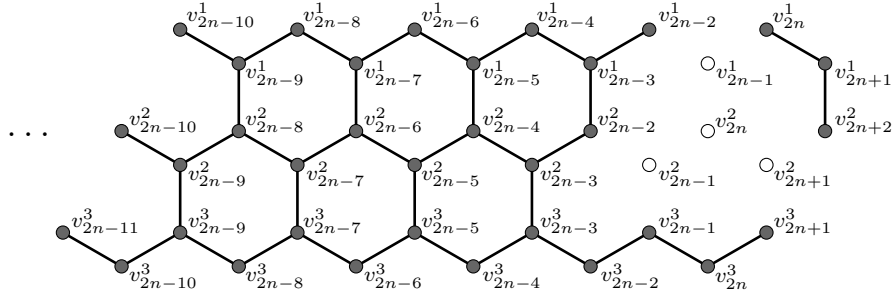

\centering

    \caption{The graph $H_n^2 \setminus  N_{H_n^2} \lsq v_{2n}^2 \rsq$.}
    \label{fig: graph of Hn2 - N[v(2n)2] in ind complex}
\end{figure}

\begin{proof}
    Observe that in $H_n^2 \setminus  N_{H_n^2} \lsq v_{2n}^2 \rsq$ (see \Cref{fig: graph of Hn2 - N[v(2n)2] in ind complex}), we fold $v_{2n}^1$ using $v_{2n+2}^2$. Similarly, we fold $v_{2n-4}^1$ and $v_{2n-2}^2$ using $v_{2n-2}^1$. On the other hand, we fold $v_{2n-1}^3$ using $v_{2n+1}^3$. At last, once $v_{2n-1}^3$ has been folded, we fold $v_{2n-3}^2$ and $v_{2n-4}^3$ using $v_{2n-2}^3$ (see \Cref{fig: foldings in Hn2 - N[v(2n)2]}).

    \begin{figure}[H]
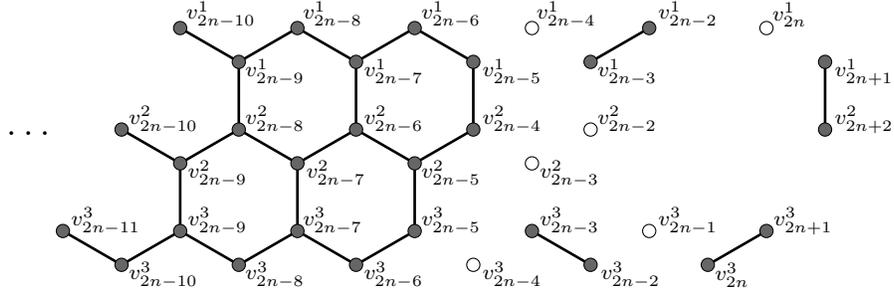

        \centering

        \caption{Foldings in the graph $H_n^2 \setminus  N_{H_n^2} \lsq v_{2n}^2 \rsq$.}
        \label{fig: foldings in Hn2 - N[v(2n)2]}
    \end{figure}

    Therefore, using the fold Lemma (\Cref{foldlemma}), we have 
    \begin{align*}
        \ind \lno H_n^2 \setminus N_{H_n^2} \lsq v_{2n}^2 \rsq \rno &\simeq \ind \lno H_{n-3}^2 \sqcup P_2 \sqcup P_2\sqcup P_2\sqcup P_2\rno\\
        &\simeq \Sigma^4 \lno \ind \lno H_{n-3}^2 \rno \rno.
    \end{align*}

    This completes the proof \Cref{lemma: link of b(2t) in Hn2}.
\end{proof}

\begin{lemma}\label{lemma: deletion of v(2n)2 in Hn2}
    For $n\geq 4$, $\ind \lno H_n^2 \setminus \lcu v_{2n}^2 \rcu \rno \simeq \Sigma^6 \lno Y_{n-3} \rno \vee \Sigma^4 \lno Y_{n-2}\rno$.
\end{lemma}

\begin{figure}[H]
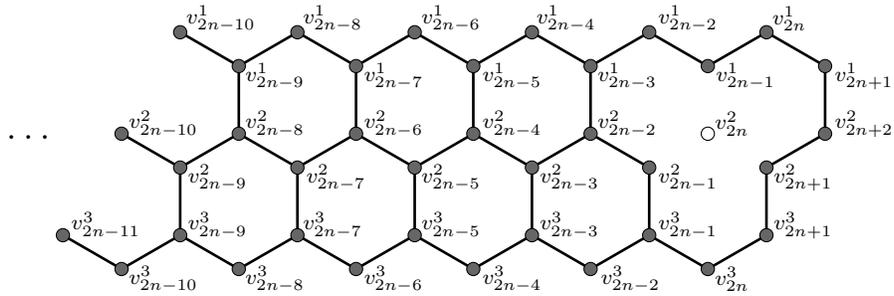

\centering

    \caption{The graph $H_n^2 \setminus \lcu v_{2n}^2 \rcu$.}
    \label{fig: graph of Hn2 - v(2n)2 in ind complex}
\end{figure}

\begin{proof}
    Let $T_n^{(1)} = H_n^2 \setminus \lcu v_{2n}^2 \rcu$ (see \Cref{fig: graph of Hn2 - v(2n)2 in ind complex}). We first discuss the homotopy type of

    \begin{enumerate}
        \item $\ind \lno T_n^{(1)} \setminus  N_{T_n^{(1)}} \lsq v_{2n-1}^3 \rsq \rno$; and,
        \item $\ind \lno T_n^{(1)} \setminus \lcu v_{2n-1}^3 \rcu \rno$.
    \end{enumerate}

    Later, we show that the inclusion map $$\ind \lno T_n^{(1)} \setminus  N_{T_n^{(1)}} \lsq v_{2n-1}^3 \rsq \rno \hookrightarrow \ind \lno T_n^{(1)} \setminus \lcu v_{2n-1}^3 \rcu \rno,$$ is null-homotopic.

    \caseheading{Homotopy type of $\bm{{\ind} \lno T_{n}^{(1)} \setminus  N_{T_{n}^{(1)}} \lsq v_{2n-1}^3 \rsq \rno}$:}

    \begin{figure}[H]
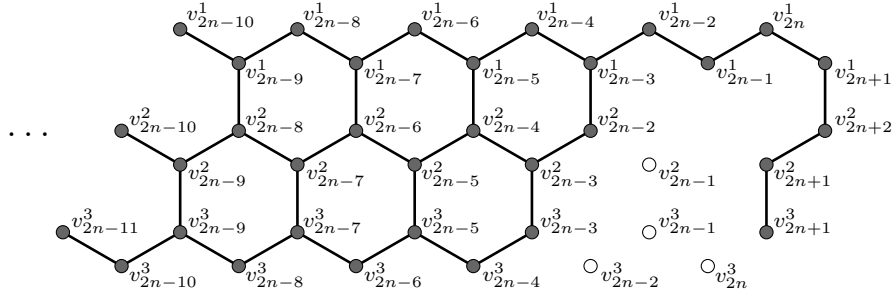

    \centering

        \caption{The graph $T_{n}^{(2)} =  T_n^{(1)} \setminus  N_{T_n^{(1)}} \lsq v_{2n-1}^3 \rsq$.}
        \label{fig: graph of Tn2 in ind complex}
    \end{figure}
    
    Let $T_{n}^{(2)} =  T_n^{(1)} \setminus  N_{T_n^{(1)}} \lsq v_{2n-1}^3 \rsq$ (see \Cref{fig: graph of Tn2 in ind complex}). Observe that in $T_{n}^{(2)}$, we fold $v_{2n+2}^2$ using $v_{2n+1}^3$. Once $v_{2n+2}^2$ is folded, we fold $v_{2n-1}^1$ using $v_{2n+1}^1$. Finally, after the folding of $v_{2n-1}^1$, we fold $v_{2n-4}^1$ and $v_{2n-2}^2$ using $v_{2n-2}^1$.

    \begin{figure}[H]
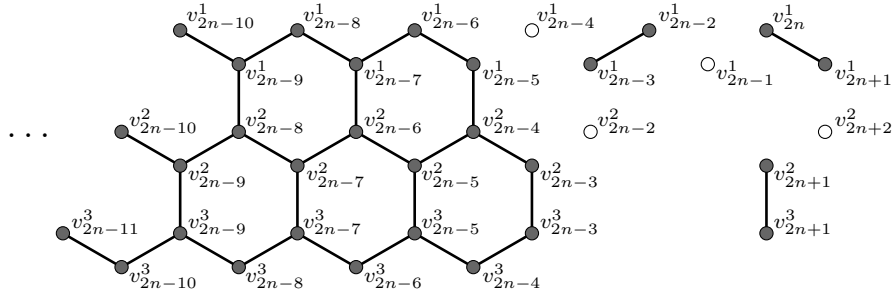

    \centering
        %
        \caption{Foldings in the graph $T_{n}^{(2)}$.}
        \label{c}
    \end{figure}

    Therefore, using the fold Lemma (\Cref{foldlemma}), we have 
    \begin{align*}
        \ind \lno T_n^{(2)} \rno &\simeq \ind \lno Y_{n-2} \sqcup P_2 \sqcup P_2 \sqcup P_2\rno\\
        &\simeq \Sigma^3 \lno \ind \lno Y_{n-2} \rno \rno.
    \end{align*}
    Hence, 
    \begin{align}\label{eqn: link of c(2t-1) in Tn(1)}
        \ind \lno T_n^{(1)} \setminus  N_{T_n^{(1)}} \lsq v_{2n-1}^3 \rsq \rno \simeq \Sigma^3 \lno \ind \lno Y_{n-2}^2 \rno \rno.
    \end{align}

    \caseheading{Homotopy type of $\bm{{\ind} \lno T_{n}^{(1)} \setminus \lcu v_{2n-1}^3 \rcu \rno}$:}

    \begin{figure}[H]
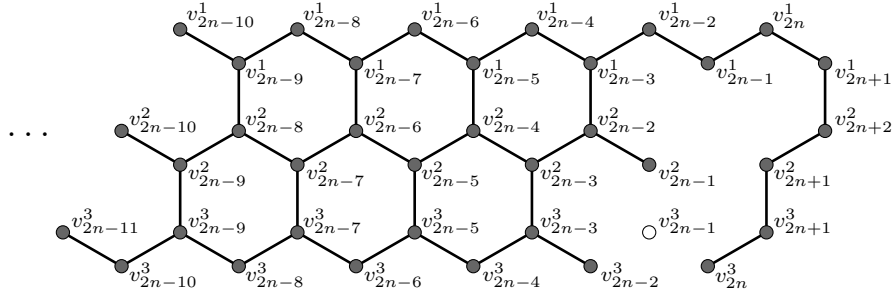

    \centering

        \caption{The graph $T_n^{(3)} = T_n^{(1)} \setminus \lcu v_{2n-1}^3 \rcu$.}
        \label{fig: graph of Tn3 in ind complex}
    \end{figure}

    Let $T_n^{(3)} = T_n^{(1)} \setminus \lcu v_{2n-1}^3 \rcu$ (see \Cref{fig: graph of Tn3 in ind complex}). Observe that in $T_n^{(3)}$, we fold $v_{2n-3}^2$ and $v_{2n-4}^3$ using $v_{2n-2}^3$. Once $v_{2n-3}^2$ has been folded, we fold $v_{2n-3}^1$ using $v_{2n-1}^2$. Following the folding of $v_{2n-3}^1$, we fold $v_{2n-6}^1$ and $v_{2n-4}^2$ using $v_{2n-4}^1$. Also, after the folding of $v_{2n-3}^1$, we fold $v_{2n}^1$ using $v_{2n-2}^1$. At last, once $v_{2n}^1$ has been folded, we fold $v_{2n+1}^2$ using $v_{2n+1}^1$ (see \Cref{fig: foldings in Tn3}).

    \begin{figure}[H]
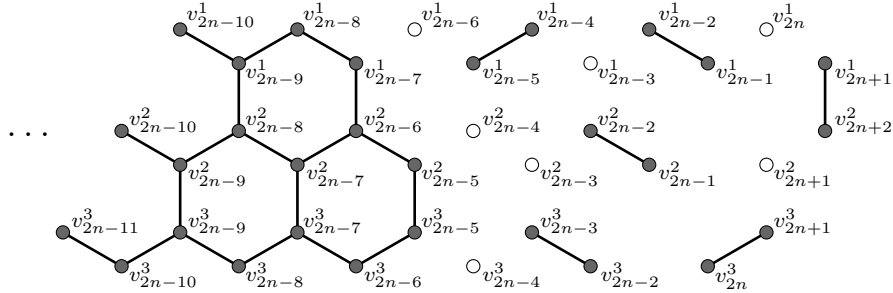

    \centering

        \caption{Foldings in the graph $T_n^{(3)} = T_n^{(1)} \setminus \lcu v_{2n-1}^3 \rcu$.}
        \label{fig: foldings in Tn3}
    \end{figure}

    Therefore, using the fold Lemma (\Cref{foldlemma}), we have
    \begin{align*}
        \ind \lno T_n^{(3)} \rno &\simeq \ind \lno Y_{n-3} \sqcup P_2 \sqcup P_2 \sqcup P_2\sqcup P_2 \sqcup P_2\sqcup P_2\rno\\
        &\simeq \Sigma^6 \lno \ind \lno Y_{n-3} \rno \rno.
    \end{align*}

     Hence,
     \begin{align}\label{eqn: deletion of c(2t-1) in Tn(1)}
         \ind \lno T_n^{(1)} \setminus \lcu v_{2n-1}^3 \rcu \rno \simeq \Sigma^6 \lno \ind \lno Y_{n-3} \rno \rno.
     \end{align}

    \caseheading{The inclusion map $\bm{{\ind} \lno T_{n}^{(1)} \setminus  N_{T_{n}^{(1)}} \lsq v_{2n-1}^3 \rsq \rno} \hookrightarrow \bm{{\ind} \lno T_{n}^{(1)} \setminus \lcu v_{2n-1}^3 \rcu \rno}$  is null-homotopic:}

    To prove this, we consider the induced subgraph $T_n^{(4)} = T_n^{(1)} \setminus \lcu v_{2n-1}^3,\ v_{2n}^3 \rcu$, that is, $T_n^{(4)}$ is precisely $T_n^{(2)} \cup \lcu v_{2n-1}^2, v_{2n-2}^3 \rcu$ (see \Cref{fig: graph of Tn4 in ind complex}).

    \begin{figure}[H]
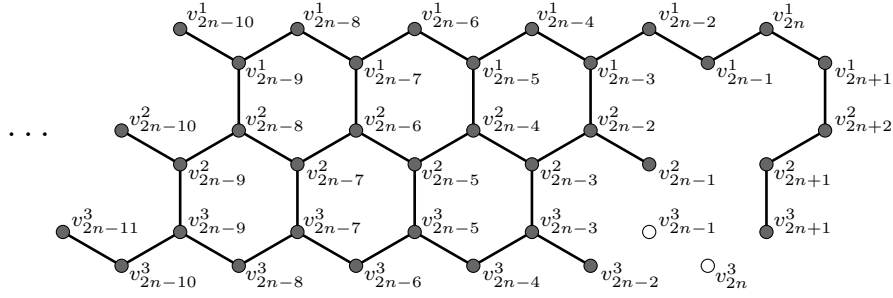

    \centering

        \caption{The graph $T_n^{(4)} = T_n^{(1)} \setminus \lcu v_{2n-1}^3,\ v_{2n}^3 \rcu$.}
        \label{fig: graph of Tn4 in ind complex}
    \end{figure}
    
    We then have the following inclusion map, 
    $$\ind \lno T_n^{(1)} \setminus  N_{T_n^{(1)}} \lsq v_{2n-1}^3 \rsq \rno \hookrightarrow \ind \lno T_n^{(4)} \rno \hookrightarrow \ind \lno T_n^{(1)} \setminus \lcu v_{2n-1}^3 \rcu \rno.$$

    The same procedure used above to determine the homotopy of $\ind \lno T_n^{(1)} \setminus \lcu v_{2n-3}^1 \rcu \rno$ will work here too. The only exception is that here, this procedure leaves $v_{2n+1}^3$ as an isolated vertex (see \Cref{fig: foldings in Tn4}).
    
    \begin{figure}[H]
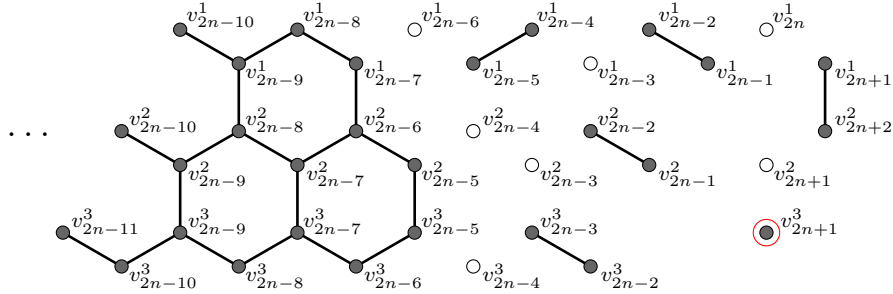

    \centering

        \caption{Foldings in the graph $T_n^{(4)}$.}
        \label{fig: foldings in Tn4}
    \end{figure}

    Thus, using the fold Lemma (\Cref{foldlemma}), we have
    \begin{align*}
        \ind \lno T_n^{(4)} \rno &\simeq \ind \lno Y_{n-3} \sqcup P_2 \sqcup P_2 \sqcup P_2 \sqcup P_2 \sqcup P_2 \sqcup \lcu v_{2n+1}^3 \rcu \rno\\
        &\simeq \lno \Sigma^5 \lno \ind \lno Y_{n-3} \rno \rno \rno * \lcu v_{2n+1}^3 \rcu\\
        &\simeq pt.
    \end{align*}

    In other words, $\ind \lno T_n^{(4)} \rno$ is contractible. Together with \Cref{prop: null-homotopy through factor}, this implies that the inclusion map,,
    $$\ind \lno T_n^{(1)} \setminus  N_{T_n^{(1)}} \lsq v_{2n-1}^3 \rsq \rno \hookrightarrow \ind \lno T_n^{(4)} \rno,$$
    is null-homotopic. Which in turn implies that, $$\ind \lno T_n^{(1)} \setminus  N_{T_n^{(1)}} \lsq v_{2n-1}^3 \rsq \rno \hookrightarrow \ind \lno T_n^{(1)} \setminus \lcu v_{2n-1}^3\rcu \rno,$$ is null-homotopic. Consequently, using \Cref{lemma: finding homotopy using link and deletion ind complex}, \Cref{eqn: link of c(2t-1) in Tn(1)} and \Cref{eqn: deletion of c(2t-1) in Tn(1)}, we have
    \begin{align*}
        \ind\lno T_n^{(1)} \rno &\simeq \ind \lno T_n^{{(1)}} \setminus \lcu v_{2n-1}^3 \rcu \rno \vee \Sigma \lno \ind \lno T_n^{(1)} \setminus  N_{T_n^{(1)}}\lsq v_{2n-1}^3 \rsq \rno \rno\\
        &\simeq \Sigma^6 \lno \ind \lno Y_{n-3} \rno \rno \vee \Sigma\lno \Sigma^3 \lno \ind \lno Y_{n-2} \rno \rno \rno\\
        &\simeq \Sigma^6 \lno \ind \lno Y_{n-3} \rno \rno \vee \Sigma^4 \lno \ind \lno Y_{n-2} \rno \rno.
    \end{align*}

    Hence, $$\ind \lno H_n^2 \setminus \lcu v_{2n}^2 \rcu \rno \simeq \Sigma^6 \lno \ind \lno Y_{n-3} \rno \rno \vee \Sigma^4 \lno \ind \lno Y_{n-2} \rno \rno.$$

    This completes the proof of \Cref{lemma: deletion of v(2n)2 in Hn2}.
\end{proof}

We are now ready to determine the homotopy type of the independence complex of double hexagonal line tiling, that is, $\ind \lno H_n^2 \rno$. Recall that here $H_n^2 = H_{1 \times 2 \times n}$.

\begin{theorem}\label{thm: homotopy type of Hn2}
    For $n\geq 1$, the independence complex of the double hexagonal line tiling, that is, $\ind \lno H_n^2 \rno$, has the following homotopy type, 
    $$\ind \lno H_n^2 \rno \simeq \begin{cases}
            \mathbb{S}^2 \vee \mathbb{S}^2, & \text{if } n=1;\\
            \mathbb{S}^4 \vee \mathbb{S}^4 \vee \mathbb{S}^4, & \text{if } n=2;\\
            \mathbb{S}^5 \vee \mathbb{S}^6, & \text{if } n=3;\\
            \Sigma^6 \lno \ind \lno Y_{n-3} \rno \rno \vee \Sigma^4 \lno \ind \lno Y_{n-2} \rno \rno \vee \Sigma^5 \lno \ind \lno H_{n-3}^2 \rno \rno, & \text{if } n \geq 4.
        \end{cases}$$
\end{theorem}

\begin{proof}
    For the case when $n=1$, the graph $H_2^1$ is isomorphic to the graph $H_{1 \times 1 \times 2}$ (see \Cref{fig: graph of H1n in ind complex}), and hence, the result follows from \Cref{thm: homotopy type of H11n}. For the remaining cases, the argument given for the initial cases in the proof of \Cref{thm: homotopy type of Yn in Hn2} also applies here.

    For $n \geq 4$, observe that it suffices to prove that the inclusion map $$\ind \lno H_n^2 \setminus  N_{H_n^2} \lsq v_{2n}^2 \rsq \rno \hookrightarrow \ind \lno H_n^2 \setminus \lcu v_{2n}^2 \rcu \rno$$ is null-homotopic, since the result then follows from \Cref{lemma: finding homotopy using link and deletion ind complex}, \Cref{lemma: link of b(2t) in Hn2}, and \Cref{lemma: deletion of v(2n)2 in Hn2}.

    For this, we consider the induced subgraph $T_n^{(5)} = H_n^2 \setminus \lcu v_{2n-1}^2, v_{2n}^2, v_{2n+1}^2 \rcu$ (see \Cref{fig: graph of Tn5 in ind complex}).
    
    \begin{figure}[H]
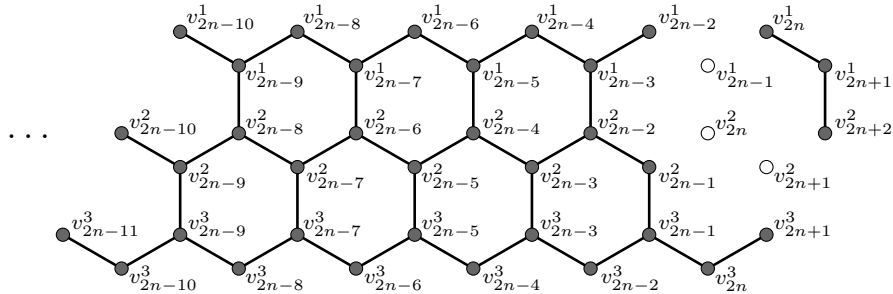

    \centering

        \caption{The graph $T_n^{(5)} = H_n^2 \setminus \lcu v_{2n-1}^2, v_{2n}^2, v_{2n+1}^2 \rcu$.}
        \label{fig: graph of Tn5 in ind complex}
    \end{figure}
    
    We then have the following inclusion map, 
    $$\ind \lno H_n^2 \setminus  N_{H_n^2} \lsq v_{2n}^2 \rsq \rno \hookrightarrow \ind \lno T_n^{(5)} \rno \hookrightarrow \ind \lno H_n^2\setminus \lcu v_{2n}^2 \rcu \rno.$$

    The same procedure used to determine the homotopy type of $\ind \lno H_n^2 \setminus  N_{H_n^2} \lsq v_{2n}^2 \rsq \rno$ in the proof of \Cref{lemma: link of b(2t) in Hn2} applies here as well. The only exception is that here, this procedure leaves $v_{2n-1}^2$ as an isolated vertex (see \Cref{fig: foldings in Tn5}).
    
    \begin{figure}[H]
    \centering

        \caption{Foldings in the graph $H_n^2 \setminus  N_{H_n^2} \lsq v_{2n}^2 \rsq$.}
        \label{fig: foldings in Tn5}
    \end{figure}
    
    Thus, using the fold Lemma (\Cref{foldlemma}), we have
    \begin{align*}
        \ind \lno T_n^{(5)} \rno &\simeq \ind \lno H_{n-3}^2 \sqcup P_2 \sqcup P_2 \ \sqcup P_2 \sqcup \lcu v_{2n-1}^2 \rcu \rno\\
        &\simeq \lno \Sigma^4 \lno \ind \lno H_{n-3}^2 \rno \rno \rno * \lcu v_{2n-1}^2 \rcu\\
        &\simeq pt.
    \end{align*}

    In other words, $\ind \lno T_n^{(5)} \rno$ is contractible. Together with \Cref{prop: null-homotopy through factor}, this implies that the inclusion map,
    $$\ind \lno H_n^2 \setminus  N_{H_n^2} \lsq v_{2n}^2 \rsq \rno \hookrightarrow \ind \lno Y_n^{(5)} \rno,$$
    is null-homotopic. Which in turn implies that, $$\ind \lno H_n^2 \setminus  N_{H_n^2} \lsq v_{2n}^2 \rsq \rno \hookrightarrow \ind \lno H_n^2 \setminus \lcu v_{2n}^2\rcu \rno,$$
    is null-homotopic. This completes the proof \Cref{thm: homotopy type of Hn2}.
\end{proof}
\section{Independence Complex of $H_{1\times 3\times n}$}

In this section, we determine the homotopy type of the independence complex of the $\lno {1\times 3 \times n} \rno$-hexagonal grid graph (see \Cref{fig: graph of H(1_3_n) in ind complex}), $H_{1\times 3 \times n}$ (or, the \textit{triple hexagonal line tiling}), that is, $\ind \lno H_{1 \times 3 \times n}\rno$, for $n \geq 1$. For simplicity, we write $H_n^3 = H_{1 \times 3 \times n}$.

\begin{figure}[H]
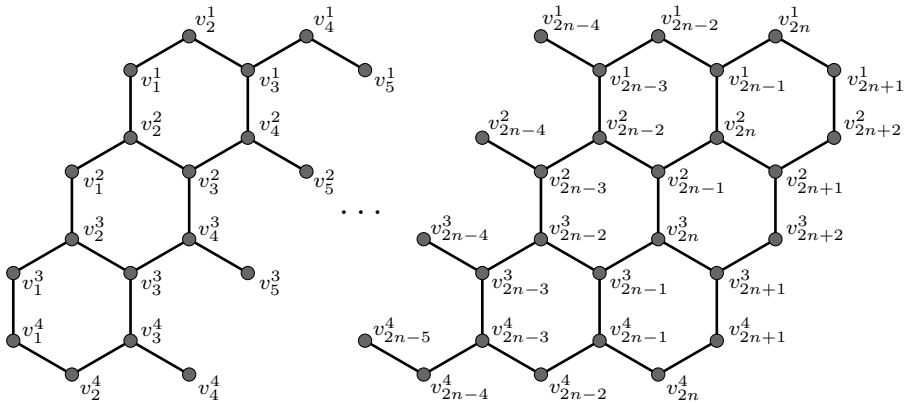

\centering

    \caption{The graph $H_{1 \times 3 \times n}$.}
    \label{fig: graph of H(1_3_n) in ind complex}
\end{figure}
    
Following induced subgraphs of $H_n^3$ will be helpful in computing the homotopy type of $\ind \lno H_n^3 \rno$.

\begin{enumerate}
    \item $Z_n^{(1)} = H_n^3 \setminus \lcu v_{2n}^1, v_{2n+1}^1, v_{2n+2}^2 \rcu$ (see \Cref{fig: graph of Zn1 in ind complex}).
    \begin{figure}[H]
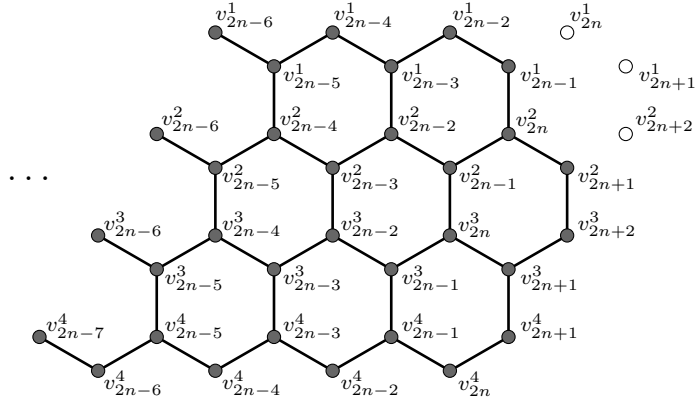

    \centering
        %
        \caption{The graph $Z_n^{(1)}$.}
        \label{fig: graph of Zn1 in ind complex}
    \end{figure}

    \item $Z_n^{(2)} = H_{n+1}^3 \setminus \lcu v_{2n+2}^1, v_{2n+3}^1, v_{2n+4}^2, v_{2n+3}^3, v_{2n+4}^3, v_{2n+2}^4, v_{2n+3}^4 \rcu $ (see \Cref{fig: graph of Zn2 in ind complex}).
    \begin{figure}[H]
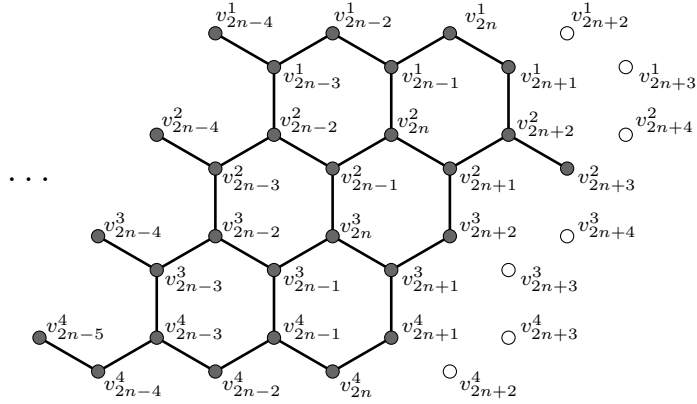

    \centering
        %
        \caption{The graph $Z_n^{(2)}$.}
        \label{fig: graph of Zn2 in ind complex}
    \end{figure}

    \item $Z_n^{(3)} = H_{n}^3 \setminus \lcu v_{2n-4}^1, v_{2n-3}^1, v_{2n-2}^1, v_{2n-1}^1, v_{2n}^1, v_{2n+1}^1, v_{2n-2}^2, v_{2n+2}^2 \rcu$ (see \Cref{fig: graph of Zn3 in ind complex}).
    
    \begin{figure}[H]
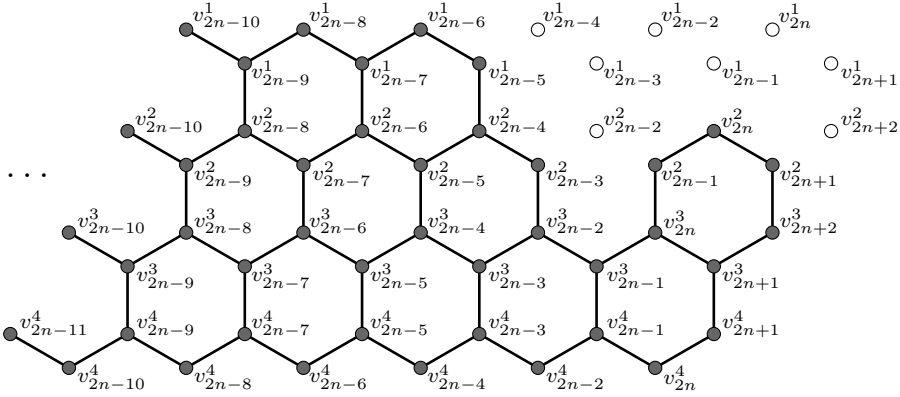

    \centering
        %
        \caption{The graph $Z_n^{(3)}$.}
        \label{fig: graph of Zn3 in ind complex}
    \end{figure}

    Note that the description of $Z_n^{(3)}$ given above handles the case $n \geq 3$. For $n = 1,2$, we define $Z_n^{(3)}$ in an analogous manner as follows.
    \begin{align*}
        Z_1^{(3)} &= H_{1}^3 \setminus \lcu v_{1}^1, v_{2}^1, v_{3}^1, v_{4}^2 \rcu;\\
        Z_2^{(3)} &= H_{2}^3 \setminus \lcu v_{1}^1, v_{2}^1, v_{3}^1, v_{4}^1, v_{5}^1, v_{2}^2, v_{6}^2 \rcu.
    \end{align*}
    
\end{enumerate}

We begin with a brief outline of the steps involved in determining the homotopy type of $\ind \lno H_n^3 \rno$, which will guide the computations that follow.

\caseheading{Outline of the procedure:}

\begin{enumerate}
    \item We first analyze the link and deletion of the vertex $v_{2n}^2$ in $\ind \lno H_n^3 \rno$, namely $\ind \lno H_n^3 \setminus N \lno v_{2n}^2 \rno \rno$ and $\ind \lno H_n^3 \setminus \lcu v_{2n}^2 \rcu \rno$. Determining the homotopy type of $\ind \lno H_n^3 \setminus N \lno v_{2n}^2 \rno \rno$ (see \Cref{lemma: link of v(2n)2 in Hn3}) 
    requires the homotopy type of $\ind \lno Z_n^{(1)} \rno$. Likewise, computing the homotopy type of $\ind \lno H_n^3 \setminus \lcu v_{2n}^2 \rcu \rno$ (see \Cref{lemma: deletion of v(2n)2 in Hn3}) requires the homotopy types of $\ind \lno Z_n^{(2)} \rno$ and $\ind \lno Z_n^{(3)} \rno$.

    \item The homotopy type of $\ind \lno Z_n^{(2)} \rno$ admits a closed formula (see \Cref{thm: Homotopy type of ind complex of Zn2}). This formula is obtained by repeatedly applying the fold lemma together with induction on $n$.

    \item The homotopy types of $\ind \lno Z_n^{(1)} \rno$ and $\ind \lno Z_n^{(3)} \rno$ are computed using recursive relations (see \Cref{thm: Homotopy type of ind complex of Zn1} and \Cref{thm: Homotopy type of ind complex of Zn3}). These formulas express $\ind \lno Z_n^{(1)} \rno$ and $\ind \lno Z_n^{(3)} \rno$ in terms of the corresponding complexes for smaller values of $n$.
\end{enumerate}

\subsection{Independence complex of $Z_n^{(1)}$}

Here, we will determine the homotopy type of the independence complex of $Z_n^{(1)}$, denoted $\ind \lno Z_n^{(1)} \rno$. We plan to do this by looking at the link and deletion of the vertex $v_{2n-1}^2$ in $\ind \lno Z_n^{(1)} \rno$, that is, $\ind \lno Z_n^{(1)} \setminus N \lsq v_{2n-1}^2 \rsq \rno$ and $\ind \lno Z_n^{(1)} \setminus \lcu v_{2n-1}^2 \rcu \rno$, respectively, and then analyzing the inclusion maps between these two complexes.

\begin{lemma}\label{lemma: link of v(2n-1)2 in Zn1}
    For $n\geq 3$, $\ind \lno Z_n^{(1)} \setminus N_{Z_n^{(1)}} \lsq v_{2n-1}^2 \rsq \rno \simeq \Sigma^6 \lno \ind \lno Z_{n-3}^{(1)} \rno \rno$.
\end{lemma}

\begin{figure}[H]
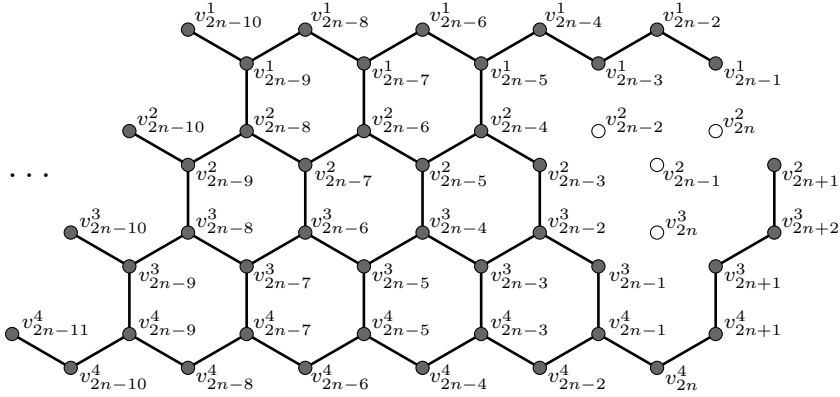

\centering

    \caption{The graph $Z_n^{(1)} \setminus N_{Z_n^{(1)}}  \lsq v_{2n-1}^2 \rsq$.}
    \label{fig: graph of Zn1 - N[v(2n-1)2] in ind complex}
\end{figure}

\begin{proof}
    In $Z_n^{(1)} \setminus N_{Z_n^{(1)}}  \lsq v_{2n-1}^2 \rsq$ (see \Cref{fig: graph of Zn1 - N[v(2n-1)2] in ind complex}), we fold $v_{2n-3}^1$ using $v_{2n-1}^1$. Following the folding of $v_{2n-3}^1$, we fold $v_{2n-6}^1$ and $v_{2n-4}^2$ using $v_{2n-4}^1$. Once $v_{2n-4}^2$ is folded, we fold $v_{2n-3}^3$ and $v_{2n-1}^3$ using $v_{2n-3}^2$. On the other hand, we fold $v_{2n+1}^3$ using $v_{2n+1}^2$. Following the folding of $v_{2n+1}^3$, we fold $v_{2n-1}^4$ using $v_{2n+1}^4$. At last, after the folding of $v_{2n-1}^4$, we fold $v_{2n-4}^4$ using $v_{2n-2}^4$ (see \Cref{fig: foldings in Zn1 - N[v(2n-1)2] in ind complex}).

    \begin{figure}[H]
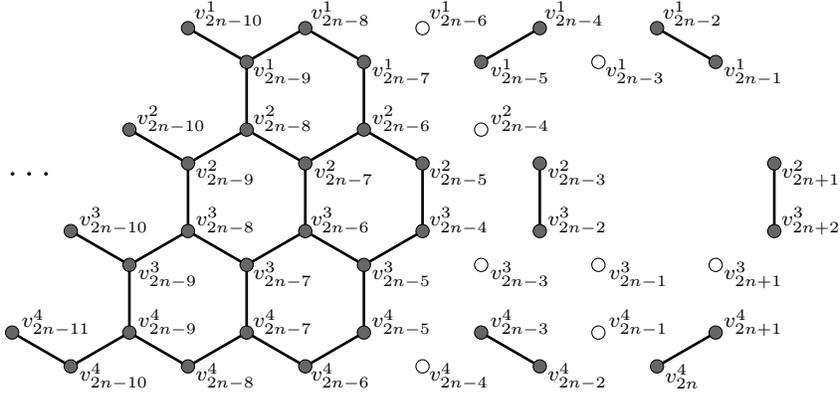

    \centering

        \caption{Foldings in the graph $Z_n^{(1)} \setminus N_{Z_n^{(1)}}  \lsq v_{2n-1}^2 \rsq$.}
        \label{fig: foldings in Zn1 - N[v(2n-1)2] in ind complex}
    \end{figure}

    Therefore, using the fold Lemma (\Cref{foldlemma}), we have 
    \begin{align*}
        \ind \lno Z_n^{(1)} \setminus N_{Z_n^{(1)}}  \lsq v_{2n-1}^2 \rsq \rno &\simeq \ind \lno Z_{n-3}^{(1)} \sqcup P_2 \sqcup P_2 \sqcup P_2 \sqcup P_2 \sqcup P_2 \sqcup P_2\rno\\
        &\simeq \Sigma^6 \lno \ind \lno Z_{n-3}^{(1)} \rno \rno.
    \end{align*}

    This completes the proof of \Cref{lemma: link of v(2n-1)2 in Zn1}.
\end{proof}

\begin{lemma}\label{lemma: deletion of v(2n-1)2 in Zn1}
    For $n\geq 2$, $$\ind \lno Z_n^{(1)} \setminus \lcu v_{2n-1}^2 \rcu \rno \simeq \Sigma^7 \lno \ind \lno Z_{n-3}^{(1)} \rno \rno \vee \Sigma^4 \lno \ind \lno Z_{n-2}^{(1)} \rno \rno.$$
\end{lemma}

\begin{figure}[H]
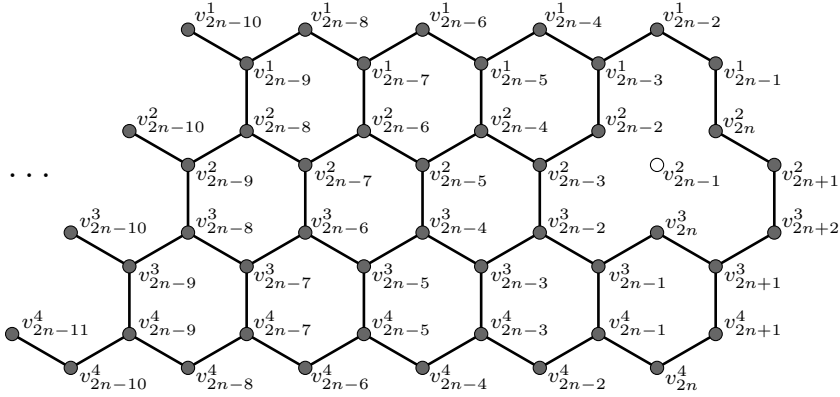

\centering

    \caption{The graph $W_n^{(1)} = Z_n^{(1)} \setminus \lcu v_{2n-1}^2 \rcu$.}
    \label{fig: graph of Zn1 - v(2n-1)2 in ind complex}
\end{figure}

\begin{proof}
    Let $W_n^{(1)} = Z_n^{(1)} \setminus \lcu v_{2n-1}^2 \rcu$ (see \Cref{fig: graph of Zn1 - v(2n-1)2 in ind complex}). We first discuss the homotopy type of

    \begin{enumerate}
        \item $\ind \lno W_n^{(1)} \setminus N_{W_n^{(1)}} \lsq v_{2n+1}^{(3)} \rsq \rno$; and,
        \item $\ind \lno W_n^{(1)} \setminus \lcu v_{2n+1}^{(3)} \rcu \rno$.
    \end{enumerate}

    Later, we show that the inclusion map $$\ind \lno W_n^{(1)} \setminus N_{W_n^{(1)}} \lsq v_{2n+1}^{(3)} \rsq \rno \hookrightarrow \ind \lno W_n^{(1)} \setminus \lcu v_{2n+1}^{(3)} \rcu \rno,$$ is null-homotopic.

    \caseheading{Homotopy type of $\bm{{\ind} \lno W_{n}^{(1)} \setminus N_{W_{n}^{(1)}} \lsq v_{2n+1}^{3} \rsq \rno}$:}

    \begin{figure}[H]
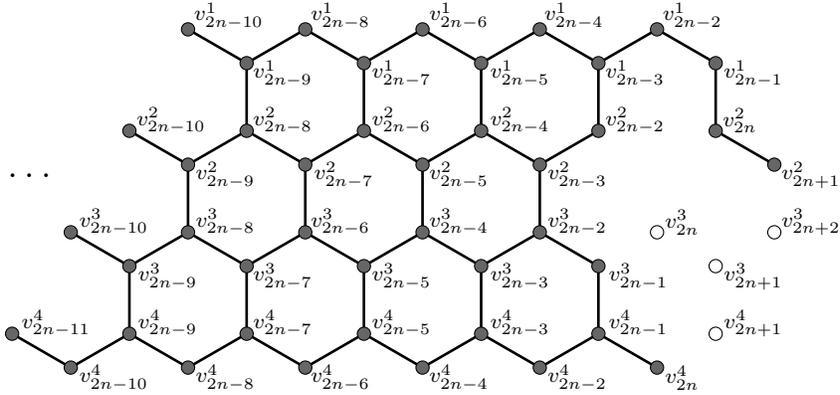

    \centering

        \caption{The graph $W_n^{(2)} = W_n^{(1)} \setminus N_{W_n^{(1)}} \lsq v_{2n+1}^{3} \rsq$.}
        \label{fig: graph of Wn2 in ind complex}
    \end{figure}
    
    Let $W_n^{(2)} = W_n^{(1)} \setminus N_{W_n^{(1)}} \lsq v_{2n+1}^{3} \rsq$ (see \Cref{fig: graph of Wn2 in ind complex}). Observe that in $W_n^{(2)}$, we fold $v_{2n-1}^3$ and $v_{2n-2}^4$ using $v_{2n}^4$. On the other hand, we fold $v_{2n-1}^1$ using $v_{2n+1}^{2}$. Finally, after the folding of $v_{2n-1}^1$, we fold $v_{2n-4}^1$ and $v_{2n-2}^2$ using $v_{2n-2}^1$ (see \Cref{fig: foldings in Wn2}).

    \begin{figure}[H]
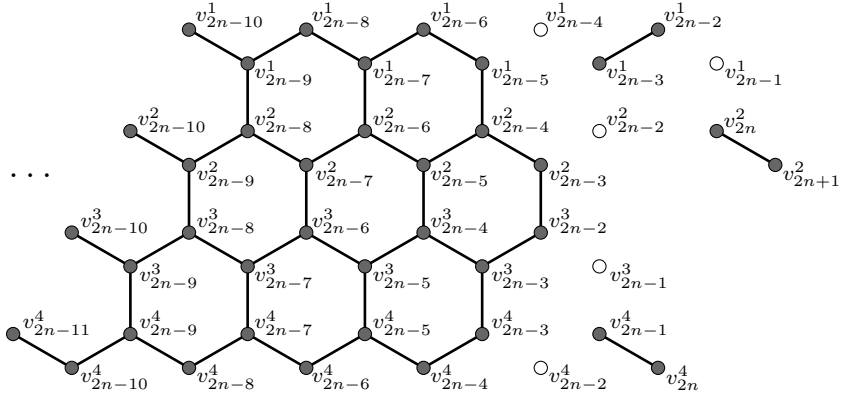

    \centering

        \caption{Foldings in $W_n^{(2)}$.}
        \label{fig: foldings in Wn2}
    \end{figure}
    
    Consequently, using the fold Lemma (\Cref{foldlemma}), we have $$\ind \lno W_n^{(2)} \rno \simeq \ind \lno Z_{n-2}^{(1)} \sqcup P_2 \sqcup P_2 \sqcup P_2 \rno \simeq \Sigma^3 \lno \ind \lno Z_{n-2}^{(1)} \rno \rno.$$

    Hence, 
    \begin{align}\label{eqn: link of v(2n+1-3) in Wn(1)}
        \ind \lno W_n^{(1)} \setminus N_{W_n^{(1)}} \lsq v_{2n+1}^{3} \rsq \rno \simeq \Sigma^3 \lno \ind \lno Z_{n-2}^{(1)} \rno \rno.
    \end{align}

    \caseheading{Homotopy type of $\bm{{\ind} \lno W_{n}^{(1)} \setminus \lcu v_{2n+1}^{3} \rcu \rno}$:}

    \begin{figure}[H]
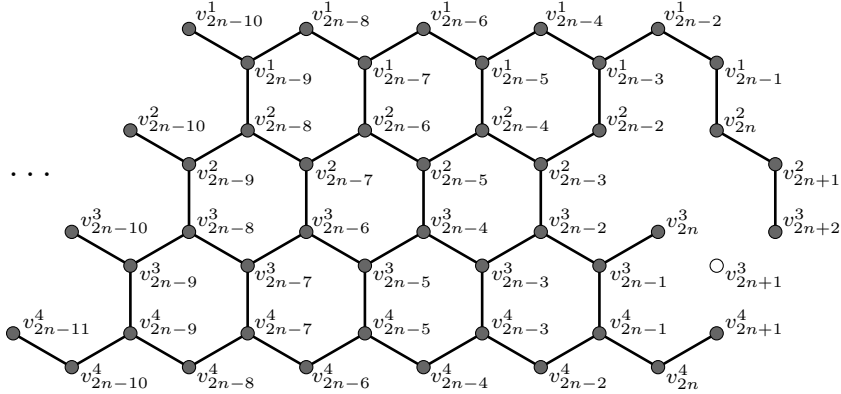

    \centering

        \caption{The graph $W_n^{(3)} = W_n^{(1)} \setminus \lcu v_{2n+1}^{3} \rcu$.}
        \label{fig: graph of Wn3 in ind complex}
    \end{figure}
    
    Let $W_n^{(3)} = W_n^{(1)} \setminus \lcu v_{2n+1}^{3} \rcu$ (see \Cref{fig: graph of Wn3 in ind complex}). Observe that in $W_n^{(3)}$, we fold $v_{2n-2}^3$ and $v_{2n-1}^4$ using $v_{2n}^3$. Following the folding of $v_{2n-1}^4$, we fold $v_{2n-3}^3$ and $v_{2n-4}^4$ using $v_{2n-2}^4$. On the other hand, we fold $v_{2n}^2$ using $v_{2n+2}^3$. Following the folding of $v_{2n}^2$, we fold $v_{2n-3}^1$ using $v_{2n-1}^1$. At last, after the folding of $v_{2n-3}^1$, we fold $v_{2n-6}^1$ and $v_{2n-4}^2$ using $v_{2n-4}^1$ (see \Cref{fig: foldings in Wn3}).

    \begin{figure}[H]
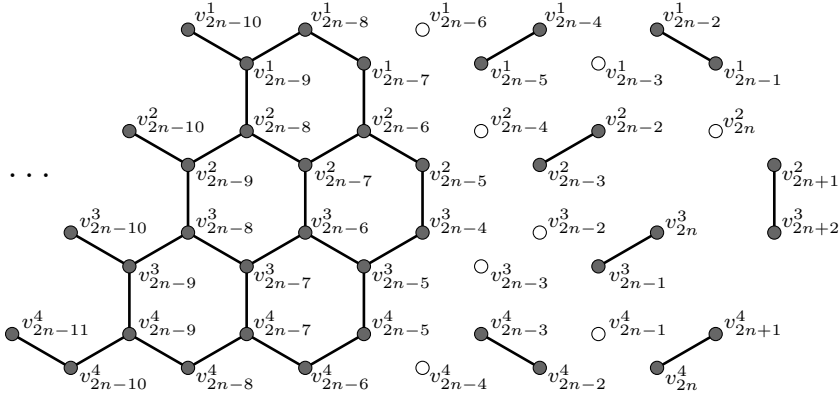

    \centering

        \caption{Foldings in $W_n^{(3)}$.}
        \label{fig: foldings in Wn3}
    \end{figure}

    Thus, using the fold Lemma (\Cref{foldlemma}), we have $$\ind \lno W_n^{(3)} \rno \simeq \ind \lno Z_{n-3}^{(1)} \sqcup P_2 \sqcup P_2 \sqcup P_2 \sqcup P_2 \sqcup P_2 \sqcup P_2 \sqcup P_2 \rno \simeq \Sigma^7 \lno \ind \lno Z_{n-3}^{(1)} \rno \rno.$$

    Hence, 
    \begin{align}\label{eqn: deletion of v(2n+1-3) in Wn(1)}
        \ind \lno W_n^{(1)} \setminus \lcu v_{2n+1}^3 \rcu\rno \simeq \Sigma^7 \lno \ind \lno Z_{n-3}^{(1)} \rno \rno.
    \end{align}

    \caseheading{The inclusion map $\bm{{\ind} \lno W_{n}^{(1)} \setminus N_{W_{n}^{(1)}} \lsq v_{2n+1}^{3} \rsq \rno \hookrightarrow {\ind} \lno W_{n}^{(1)} \setminus \lcu v_{2n+1}^{3} \rcu \rno}$ is null-homotopic:}

    To prove this, consider the induced subgraph $W_n^{(4)} = W_n^{(1)} \setminus \lcu v_{2n+1}^3, v_{2n+2}^3, v_{2n+1}^4\rcu$, that is, $W_n^{(4)}$ is precisely $W_n^{(2)} \sqcup \lcu v_{2n}^3 \rcu$ (see \Cref{fig: graph of Wn4 in ind complex}). Recall that, $W_n^{(2)} = W_n^{(1)} \setminus \lsq v_{2n+1}^3 \rsq$

    \begin{figure}[H]
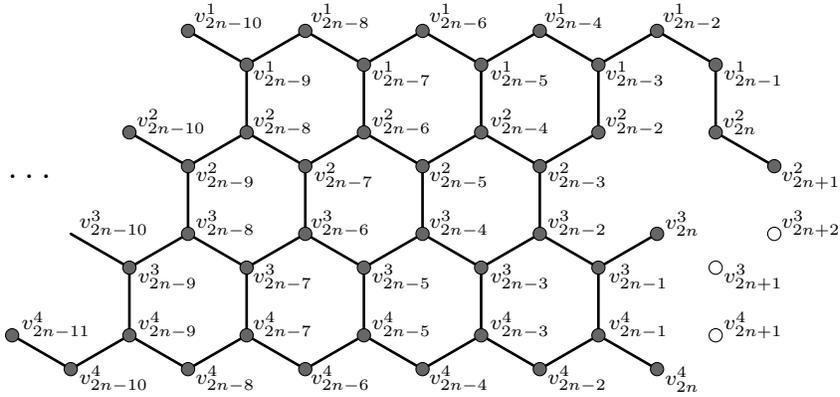

    \centering

        \caption{The graph $W_n^{(4)} = W_n^{(1)} \setminus \lcu v_{2n+1}^3, v_{2n+2}^3, v_{2n+1}^4\rcu$.}
        \label{fig: graph of Wn4 in ind complex}
    \end{figure}
    
    We then have the following inclusion map, $$\ind \lno W_n^{(1)} \setminus N_{W_n^{(1)}} \lsq v_{2n+1}^{3} \rsq \rno \hookrightarrow \ind \lno W_n^{(4)} \rno \hookrightarrow \ind \lno W_n^{(1)} \setminus \lcu v_{2n+1}^{3} \rcu \rno.$$

    The same procedure used to determine the homotopy type of $\ind \lno W_n^{(2)} \rno$ above will follow here as well. The only exception is that here, this procedure leaves $v_{2n}^3$ as an isolated vertex (see \Cref{fig: foldings in Wn4}).
    
    \begin{figure}[H]
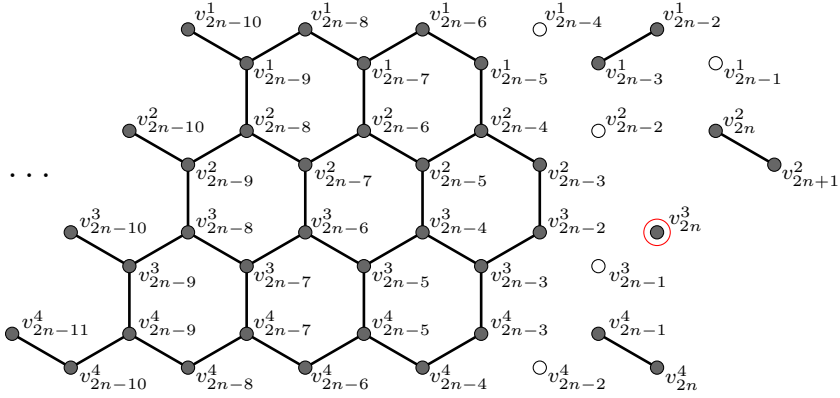

    \centering

        \caption{Foldings in $W_n^{(4)}$.}
        \label{fig: foldings in Wn4}
    \end{figure}
    
    Thus, using the fold Lemma (\Cref{foldlemma}), we have 
    \begin{align*}
        \ind \lno W_n^{(4)} \rno &\simeq \ind \lno Z_{n-2}^{(1)} \sqcup P_2 \sqcup P_2 \sqcup P_2 \sqcup \lcu v_{2n}^3 \rcu \rno \\
        & \simeq \Sigma^3 \lno \ind \lno Z_{n-2}^{(1)} \rno \rno * \lcu v_{2n}^3 \rcu\\
        & \simeq pt.
    \end{align*}

    In other words, $\ind \lno W_n^{(4)} \rno$ is contractible. Together with \Cref{prop: null-homotopy through factor}, this implies that the inclusion map, $$\ind \lno W_n^{(1)} \setminus N_{W_n^{(1)}} \lsq v_{2n+1}^{3} \rsq \rno \hookrightarrow \ind \lno W_n^{(4)} \rno,$$ is null-homotopic. Which in turn implies that the inclusion map $$\ind \lno W_n^{(1)} \setminus N_{W_n^{(1)}} \lsq v_{2n+1}^{3} \rsq \rno  \hookrightarrow \ind \lno W_n^{(1)} \setminus \lcu v_{2n+1}^{3} \rcu \rno,$$ is null-homotopic. Therefore, using \Cref{lemma: finding homotopy using link and deletion ind complex}, \Cref{eqn: link of v(2n+1-3) in Wn(1)} and \Cref{eqn: deletion of v(2n+1-3) in Wn(1)}, we have
    \begin{align*}
        \ind\lno W_n^{(1)} \rno &\simeq \ind \lno W_n^{{(1)}} \setminus \lcu v_{2t+1}^3 \rcu \rno \vee \Sigma \lno \ind \lno W_n^{(1)} \setminus N_{T_n^{(1)}}\lsq v_{2t+1}^3 \rsq \rno \rno\\
        &\simeq \Sigma^7 \lno \ind \lno Z_{n-3}^{(1)} \rno \rno \vee \Sigma\lno \Sigma^3 \lno \ind \lno Z_{n-2}^{(1)} \rno \rno\rno\\
        &\simeq \Sigma^7 \lno \ind \lno Z_{n-3}^{(1)} \rno \rno \vee \Sigma^4 \lno \ind \lno Z_{n-2}^{(1)} \rno \rno.
    \end{align*}

    Hence, $$\ind \lno Z_n^{(1)} \setminus \lcu v_{2n-1}^2 \rcu \rno \simeq \Sigma^7 \lno \ind \lno Z_{n-3}^{(1)} \rno \rno \vee \Sigma^4 \lno \ind \lno Z_{n-2}^{(1)} \rno \rno.$$

    This completes the proof of \Cref{lemma: deletion of v(2n-1)2 in Zn1}.
\end{proof}

\begin{theorem}\label{thm: Homotopy type of ind complex of Zn1}
    For $n\geq 1$,
    $$\ind \lno Z_n^{(1)} \rno \simeq \begin{cases}
                \mathbb{S}^2, & \text{if } n=1;\\
                \mathbb{S}^5 \vee \mathbb{S}^5, & \text{if } n=2;\\
                \mathbb{S}^6, & \text{if } n=3;\\
                \Sigma^{7}\lno \ind \lno Z_{n-3}^{(1)} \rno \rno \vee \Sigma^{4}\lno \ind \lno Z_{n-2}^{(1)} \rno \rno \vee \Sigma^{7}\lno \ind \lno Z_{n-3}^{(1)} \rno \rno, & \text{if } n \geq 4.
            \end{cases}$$
\end{theorem}

\begin{proof}
    For the case $n=1$, the graph $Z_1^{(1)}$ is isomorphic to the graph $X_{3}^{(2)}$ (see \Cref{fig: graph of Xn2 - ind complex}). Therefore, the result follows immediately from \Cref{lemma: homotopy of Xn2}. The cases $n=2,3$ follow by the same reasoning given for the initial cases as in \Cref{thm: homotopy type of Yn in Hn2}.
    
    For $n \geq 4$. observe that it suffices to prove that the inclusion map $$\ind \lno Z_n^{(1)} \setminus N \lsq v_{2n-1}^2 \rsq \rno  \hookrightarrow \ind \lno Z_n^{(1)} \setminus \lcu v_{2n-1}^2 \rcu \rno $$ is null-homotopic, since the result then follows from \Cref{lemma: finding homotopy using link and deletion ind complex}, \Cref{lemma: link of v(2n-1)2 in Zn1}, and \Cref{lemma: deletion of v(2n-1)2 in Zn1}.

    For this, we consider the induced subgraph $W_n^{(5)} = Z_n^{(1)} \setminus \lcu v_{2n-2}^2, v_{2n-1}^2, v_{2n}^2\rcu$.
    
    \begin{figure}[H]
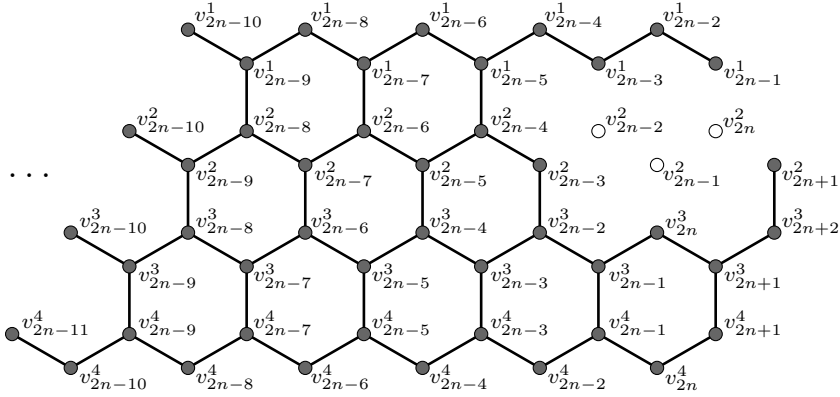

    \centering

        \caption{The graph $W_n^{(5)} = Z_n^{(1)} \setminus \lcu v_{2n-2}^2, v_{2n-1}^2, v_{2n}^2\rcu$.}
        \label{fig: graph of Wn5}
    \end{figure}

    We then have the following inclusion map, $$\ind \lno Z_n^{(1)} \setminus N_{Z_n^{(1)}} \lsq v_{2n-1}^2 \rsq \rno \hookrightarrow \ind \lno W_n^{(5)} \rno \hookrightarrow \ind \lno Z_n^{(1)}\setminus \lcu v_{2n-1}^2 \rcu \rno.$$

    The same procedure used to determine the homotopy type of $\ind \lno Z_n^{(1)} \setminus N_{Z_n^{(1)}} \lsq v_{2n-1}^2 \rsq \rno$ in the proof of \Cref{lemma: link of v(2n-1)2 in Zn1} applies here as well. The only exception is that here, this procedure leaves $v_{2n}^3$ as an isolated vertex (see \Cref{fig: foldings in Wn5}).

    \begin{figure}[H]
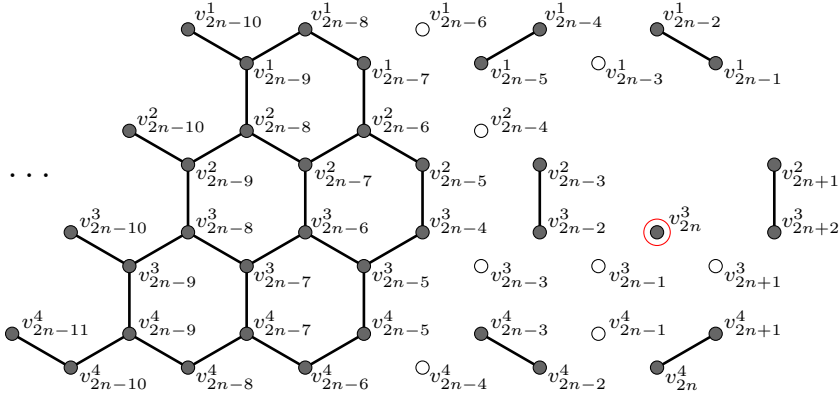

    \centering

        \caption{Foldings in the graph $W_n^{(5)}$.}
        \label{fig: foldings in Wn5}
    \end{figure}
    
    Thus, using the fold Lemma (\Cref{foldlemma}), we have
    \begin{align*}
        \ind \lno W_n^{(5)} \rno &\simeq \ind \lno Z_{n-3}^{(1)} \sqcup P_2 \sqcup P_2 \sqcup P_2 \sqcup P_2 \sqcup P_2 \sqcup P_2 \sqcup \lcu v_{2n}^3 \rcu \rno\\
        &\simeq \lno \Sigma^6 \lno Z_{n-3}^{(1)} \rno \rno * \lcu v_{2n}^3 \rcu\\
        &\simeq pt.
    \end{align*}

    In other words, $\ind \lno W_n^{(5)} \rno$ is contractible. Together with \Cref{prop: null-homotopy through factor}, this implies that the inclusion map,,
    $$\ind \lno Z_n^{(1)} \setminus N_{Z_n^{(1)}} \lsq v_{2n-1}^2 \rsq \rno \hookrightarrow \ind \lno W_n^{(5)} \rno,$$
    is null-homotopic. Which in turn implies that, $$\ind \lno Z_n^{(1)} \setminus N_{Z_n^{(1)}} \lsq v_{2n-1}^2 \rsq \rno \hookrightarrow \ind \lno Z_n^{(1)}\setminus \lcu v_{2n-1}^2 \rcu \rno,$$
    is null-homotopic. This completes the proof of \Cref{thm: Homotopy type of ind complex of Zn1}.
\end{proof}

\subsection{Independence complex of $Z_n^{(2)}$}

Here, we will determine the homotopy type of the independence complex of $Z_n^{(2)}$, denoted $\ind \lno Z_n^{(2)} \rno$. Recall that, $$Z_n^{(2)} = H_{n+1}^3 \setminus \lcu v_{2n+2}^1, v_{2n+3}^1, v_{2n+4}^2, v_{2n+3}^3, v_{2n+4}^3, v_{2n+2}^4, v_{2n+3}^4 \rcu,$$ where $n\geq 1$. The main result is stated in the following theorem.

\begin{theorem}\label{thm: Homotopy type of ind complex of Zn2}
    For $n \geq 1$, the independence complex of $Z_n^{(2)}$ is contractible, when $n$ is odd and is homotopy equivalent to a $\lno 2n + 1 \rno-$dimensional sphere, when $n$ is even. In other words, we have

    $$\ind \lno Z_n^{(2)} \rno \simeq \begin{cases}
         pt, & \text{if } n \text{ is odd };\\
         \bbS^{2n+1}, & \text{if } n \text{ is even }.
    \end{cases}$$
\end{theorem}

\begin{proof}
    We will prove this separately for the cases when $n$ is odd and when $n$ is even. We will use the induction on $n$, along with the fold lemma.

    \caseheading{When $\bm{n}$ is odd:}
    
    For the base case, that is, $n=1$, note that, we fold $v_3^1$ and $v_3^2$ using $v_5^2$ in $Z_1^{(2)}$. Following the folding of $v_3^2$, we fold $v_2^3$ and $v_3^4$ using $v_4^3$. This procedure leaves us with two induced $P_2$, and one induced $P_3$ and $P_4$ each (see \Cref{fig: foldings in Z1(2)}).
    \begin{figure}[H]
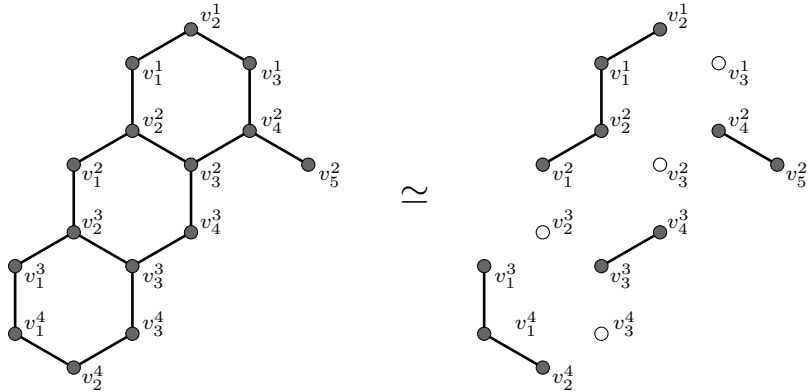

    \centering

        \caption{Foldings in the graph $Z_1^{(2)}$.}
        \label{fig: foldings in Z1(2)}
    \end{figure}
    
    Thus, using the fold Lemma (\Cref{foldlemma}) and \Cref{thm: homotopy type of ind complex of path graph}, we have
    \begin{align*}
        \ind \lno Z_1^{(2)} \rno & \simeq \ind \lno P_4 \sqcup P_3 \sqcup P_2 \sqcup P_2 \rno\\
        &\simeq \Sigma^2 \lno \ind \lno P_3 \rno \rno * \ind \lno P_4 \rno\\
        &\simeq pt.
    \end{align*}
    Thus, the result holds for $n=1$. Assuming the given result is true for all odd $n \leq 2t+1$, we aim to prove it for $n=2t+3$.

    In $Z_{2t+3}^{(2)}$, we fold $v_{4t+7}^1$ and $v_{4t+7}^2$ using $v_{4t+9}^2$. Following the folding of $v_{4t+7}^1$, we fold $v_{4t+4}^1$ and $v_{4t+6}^2$ using $v_{4t+6}^1$. On the other hand, after the folding of $v_{4t+7}^2$, we fold $v_{4t+6}^3$ and $v_{4t+7}^4$ using $v_{4t+8}^3$. At last, once $v_{4t+7}^4$ is folded, we fold $v_{4t+5}^3$ and $v_{4t+4}^4$ using $v_{4t+6}^4$ (see \Cref{fig: foldings in Z(2t+3)3}).

    \begin{figure}[H]
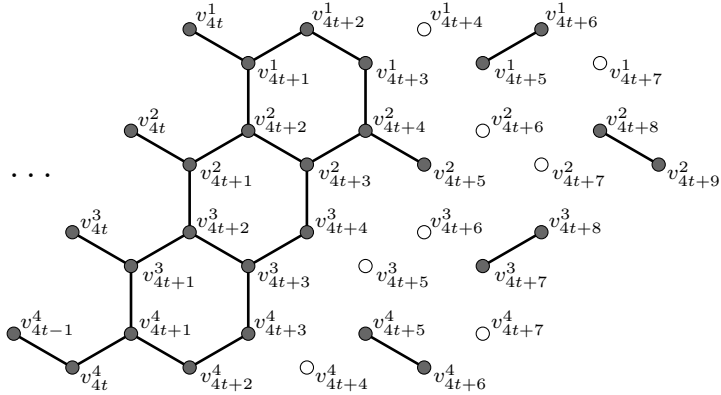

    \centering

        \caption{Foldings in the graph $Z_{2t+3}^{(3)}$.}
        \label{fig: foldings in Z(2t+3)3}
    \end{figure}

    Thus, using the fold Lemma (\Cref{foldlemma}), we have
    \begin{align*}
        \ind \lno Z_{2t+3}^{(2)} \rno &\simeq \ind\lno Z_{2t+1}^{(2)} \sqcup P_2 \sqcup P_2 \sqcup P_2 \sqcup P_2 \rno\\
        &\simeq \Sigma^4 \lno \ind \lno Z_{2t+1}^{(2)} \rno \rno\\
        &\simeq pt.
    \end{align*}

    This completes the proof for the case when $n$ is odd.

    \caseheading{When $\bm{n}$ is even:}
    
    For the base case, that is, $n=2$, note that, we fold $v_5^1$ and $v_5^2$ using $v_7^2$ in $Z_2^{(2)}$. Following the folding of $v_5^1$, we fold $v_2^1$ and $v_4^2$ using $v_4^1$. Once $v_2^1$ is folded, we fold $v_1^2$ and $v_3^2$ using $v_1^1$. On the other hand, using the folding of $v_5^2$, we fold $v_4^3$ and $v_5^4$ using $v_6^3$. At last, after the folding of $v_5^4$, we fold $v_3^3$ and $v_2^4$ using $v_4^4$. This procedure leaves us with five induced $P_2$, and one induced $P_3$ (see \Cref{fig: foldings in Z2(2)}).

    \begin{figure}[H]
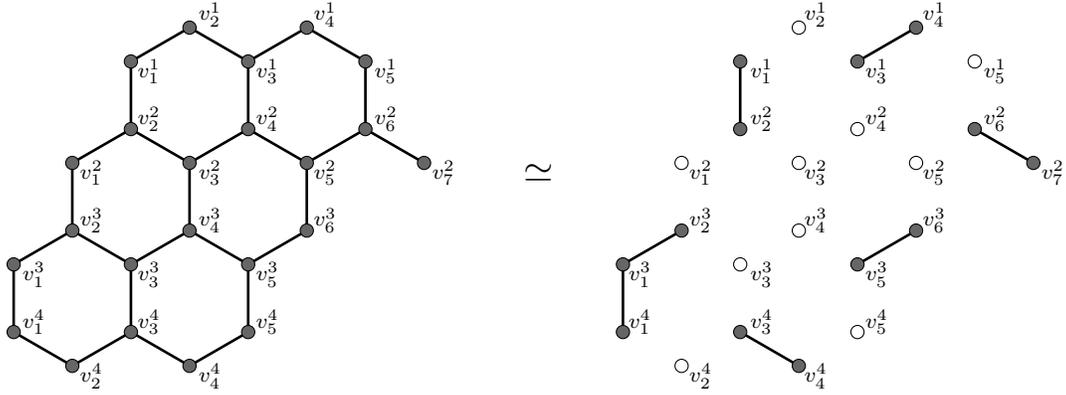

    \centering

        \caption{Foldings in the graph $Z_2^{(2)}$.}
        \label{fig: foldings in Z2(2)}
    \end{figure}
    
    Thus, using the fold Lemma (\Cref{foldlemma}) and \Cref{thm: homotopy type of ind complex of path graph}, we have
    \begin{align*}
        \ind \lno Z_2^{(2)} \rno & \simeq \ind \lno P_3 \sqcup P_2 \sqcup P_2 \sqcup P_2\sqcup P_2\sqcup P_2 \rno\\
        &\simeq \Sigma^5 \lno \ind \lno P_3 \rno \rno \\
        &\simeq \Sigma^5 \lno \bbS^0 \rno \\
        &\simeq \bbS^5.
    \end{align*}

    Thus, the result holds for $n=2$. Assuming the given result is true for all even $n \leq 2t$, we aim to prove it true for $n=2t+2$, that is, to show that $\ind \lno Z_{2t+2}^{(2)} \rno \simeq \bbS^{4t+5}$.
    
    The procedure is analogous to the one used in the case when $n$ is odd, except for minor changes in labeling (see \Cref{fig: foldings in Z(2t+2)2}).

    \begin{figure}[H]
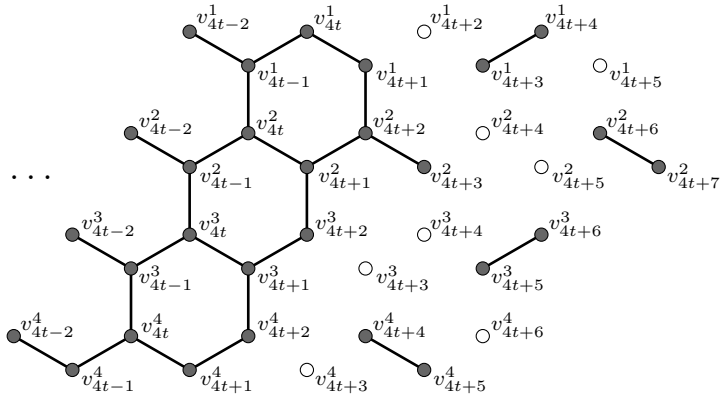

    \centering

        \caption{Foldings in the graph $Z_{2t+2}^{(2)}$.}
        \label{fig: foldings in Z(2t+2)2}
    \end{figure}
    
    Thus, using the fold Lemma (\Cref{foldlemma}), we have
    \begin{align*}
        \ind \lno Z_{2t+2}^{(2)} \rno &\simeq \ind\lno Z_{2t}^{(2)} \sqcup P_2 \sqcup P_2 \sqcup P_2 \sqcup P_2 \rno\\
        &\simeq \Sigma^4 \lno \ind \lno Z_{2t}^{(2)} \rno \rno\\
        &\simeq \Sigma^4 \lno \bbS^{4t+1}\rno\\
        &\simeq \bbS^{4t+5}.
    \end{align*}

    This completes the case when $n$ is even, thereby completing the proof of \Cref{thm: Homotopy type of ind complex of Zn2}.
\end{proof}

\subsection{Independence complex of $Z_n^{(3)}$}

Here, we will determine the homotopy type of the independence complex of $Z_{n}^{(3)}$, denoted $\ind \lno Z_{n}^{(3)} \rno$. Recall that, for $n\geq 3$, $$Z_n^{(3)} = H_{n}^3 \setminus \lcu v_{2n-4}^1, v_{2n-3}^1, v_{2n-2}^1, v_{2n-1}^1, v_{2n}^1, v_{2n+1}^1, v_{2n-2}^2, v_{2n+2}^2 \rcu.$$ Similarly, for $n=2, 3$, 
\begin{align*}
    Z_1^{(3)} &= H_{n}^3 \setminus \lcu v_{1}^1, v_{2}^1, v_{3}^1, v_{4}^2 \rcu;\\
    Z_2^{(3)} &= H_{n}^3 \setminus \lcu v_{1}^1, v_{2}^1, v_{3}^1, v_{4}^1, v_{5}^1, v_{2}^2, v_{6}^2 \rcu.
\end{align*}
We again do this by carefully looking into the link and deletion of $v_{2n}^2$ in the $\ind \lno Z_n^{(3)} \rno$, that is, $\ind \lno Z_n^{(3)} \setminus N \lsq v_{2n}^2 \rsq \rno$ and $\ind \lno Z_n^{(3)} \setminus \lcu v_{2n}^2 \rcu \rno$, respectively, and then analyzing the inclusion maps between these two complexes.

\begin{lemma}\label{lemma: link of v(2n)2 in Zn3}
    For $n\geq 3$, $\ind \lno Z_n^{(3)} \setminus N_{Z_n^{(3)}} \lsq v_{2n}^2 \rsq \rno \simeq \Sigma^2 \lno \ind \lno Z_{n-2}^{(1)} \rno \rno$.
\end{lemma}

\begin{figure}[H]
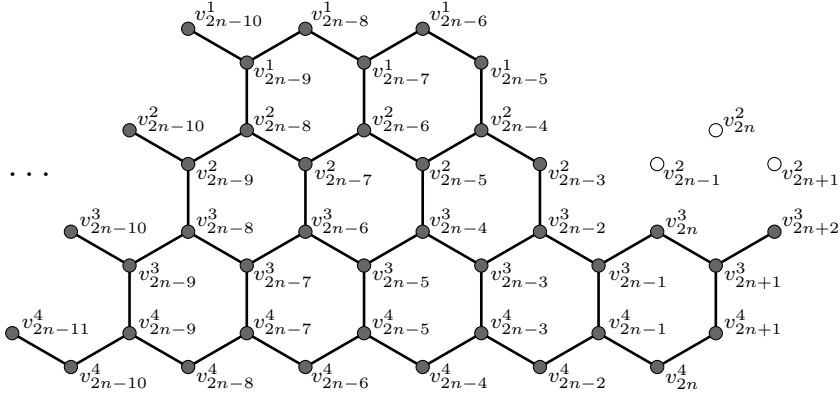

    \centering

        \caption{The graph $Z_n^{(3)} \setminus N_{Z_n^{(3)}}  \lsq v_{2n}^2 \rsq$.}
        \label{fig: graph of Zn3 - N[2n(2)] in ind complex}
    \end{figure}

\begin{proof}
    Observe that in $Z_n^{(3)} \setminus N_{Z_n^{(3)}} \lsq v_{2n}^2 \rsq$ (see \Cref{fig: graph of Zn3 - N[2n(2)] in ind complex}), we fold $v_{2n}^3$ and $v_{2n+1}^4$ using $v_{2n+2}^3$. Following the folding of $v_{2n+1}^4$, we fold $v_{2n-1}^3$ and $v_{2n-2}^4$ using $v_{2n}^4$ (see \Cref{fig: foldings in Zn3 - N[2n(2)]}).

    \begin{figure}[H]
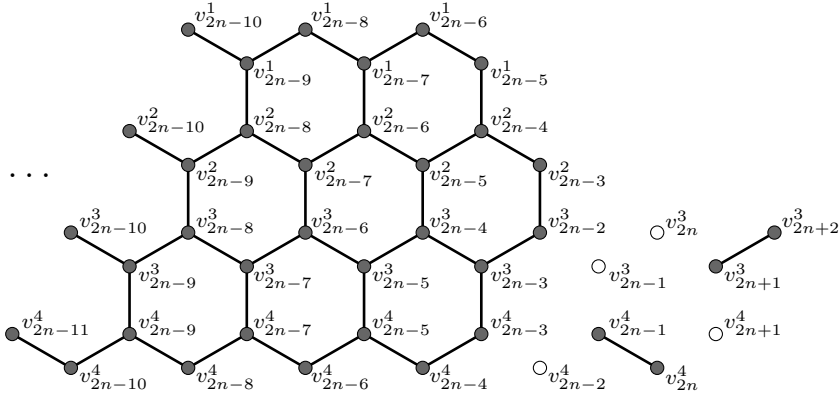

    \centering
        %
        \caption{Foldings in the graph $Z_n^{(3)} \setminus N_{Z_n^{(3)}}  \lsq v_{2n}^2 \rsq$.}
        \label{fig: foldings in Zn3 - N[2n(2)]}
    \end{figure}

    Therefore, using the fold Lemma (\Cref{foldlemma}), we have 
    \begin{align*}
        \ind \lno Z_n^{(3)} \setminus N_{Z_n^{(3)}}  \lsq v_{2n}^2 \rsq \rno &\simeq \ind \lno Z_{n-2}^{(1)} \sqcup P_2 \sqcup P_2\rno\\
        &\simeq \Sigma^2 \lno \ind \lno Z_{n-2}^{(1)} \rno \rno.
    \end{align*}

    This completes the proof of \Cref{lemma: link of v(2n)2 in Zn3}.
\end{proof}

\begin{lemma}\label{lemma: deletion of v(2n)2 in Zn3}
    For $n\geq 3$, $$\ind \lno Z_n^{(3)} \setminus \lcu v_{2n}^2 \rcu \rno \simeq \Sigma^7 \lno \ind \lno Z_{n-2}^{(3)} \rno \rno.$$
\end{lemma}

\begin{figure}[H]
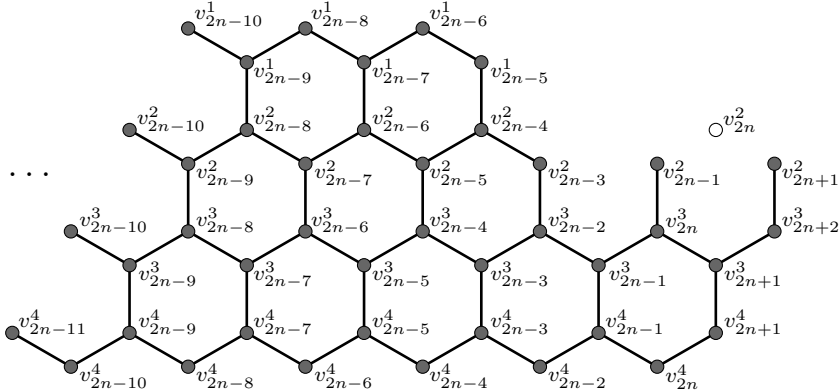

\centering

    \caption{The graph $Z_n^{(3)} \setminus \lcu v_{2n}^2 \rcu$.}
    \label{fig: graph of Zn3 - v-2n(2) in ind complex}
\end{figure}

\begin{proof}
    We perform foldings on the vertices of $Z_n^{(3)} \setminus \lcu v_{2n}^2 \rcu$ (see \Cref{fig: graph of Zn3 - v-2n(2) in ind complex}), according to the steps outlined below.
    \begin{enumerate}[label = (\roman*)]
        \item At first, we fold $v_{2n-1}^3$ and $v_{2n+1}^3$ using $v_{2n-1}^2$.
        \item Once $v_{2n+1}^3$ is folded, we fold $v_{2n-1}^4$ using $v_{2n+1}^4$.
        \item Following the folding of $v_{2n-1}^4$, we fold $v_{2n-4}^4$ and $v_{2n-3}^3$ using $v_{2n-2}^4$.
        \item After the folding of $v_{2n-3}^3$ and $v_{2n-1}^3$ (from step (i)), we fold $v_{2n-4}^2$ using $v_{2n-2}^3$.
        \item Once $v_{2n-4}^2$ is folded, we fold $v_{2n-7}^1$ using $v_{2n-5}^1$.
        \item At last, following the folding of $v_{2n-7}^1$, we fold $v_{2n-10}^1$ and $v_{2n-8}^2$ using $v_{2n-8}^1$.
    \end{enumerate}

    This procedure leaves us with seven copies of induced $P_2$ and along with $Z_{n-3}^{(3)}$ (see \Cref{fig: foldings in Zn3 - v-2n(2)}).

    \begin{figure}[H]
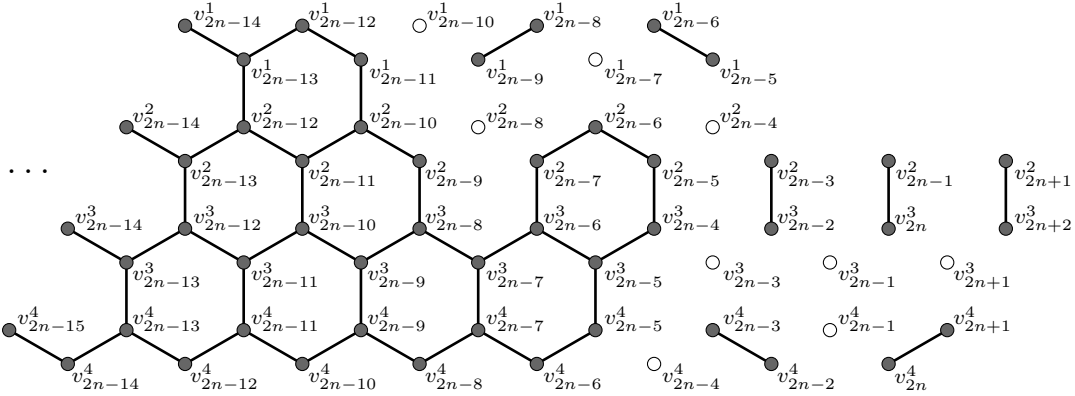

    \centering

        \caption{Foldings in the graph $Z_n^{(3)} \setminus \lcu v_{2n}^2 \rcu$.}
        \label{fig: foldings in Zn3 - v-2n(2)}
    \end{figure}

    Therefore, using the fold Lemma (\Cref{foldlemma}), we have 
    \begin{align*}
        \ind \lno Z_n^{(3)} \setminus \lcu v_{2n}^2 \rcu \rno &\simeq \ind \lno Z_{n-3}^{(3)} \sqcup P_2 \sqcup P_2 \sqcup P_2 \sqcup P_2 \sqcup P_2 \sqcup P_2 \sqcup P_2 \rno\\
        &\simeq \Sigma^7 \lno \ind \lno Z_{n-3}^{(3)} \rno \rno.
    \end{align*}

    This completes the proof of \Cref{lemma: deletion of v(2n)2 in Zn3}.
\end{proof}

We now have all the necessary results to discuss the homotopy type of $Z_{n}^{(3)}$. We state this result in the following theorem.

\begin{theorem}\label{thm: Homotopy type of ind complex of Zn3}
    For $n\geq 1$, 
        $$\ind \lno Z_{n}^{(3)} \rno \simeq \begin{cases}
            \mathbb{S}^2 \vee \mathbb{S}^2, & \text{if } n=1;\\
            \mathbb{S}^4, & \text{if } n=2;\\
            \mathbb{S}^5 \vee \mathbb{S}^5, & \text{if } n=3;\\
            \mathbb{S}^8 \vee \mathbb{S}^8 \vee \mathbb{S}^8 \vee \mathbb{S}^8, & \text{if } n=4;\\
            \Sigma^7 \lno \ind \lno Z_{n-3}^{(3)} \rno \rno \vee \Sigma^3 \lno \ind \lno Z_{n-2}^{(1)} \rno \rno, & \text{if } n \geq 5.
        \end{cases}$$
\end{theorem}

\begin{proof}
    When $n=1$, the graph $Z_1^{(3)}$ here is isomorphic to $H_{1 \times 1 \times 2}$ (see \Cref{fig: graph of H1n in ind complex}) and the result then follows from \Cref{thm: homotopy type of H11n}. All the remaining cases follow the same argument given for the initial cases in the proof of \Cref{thm: homotopy type of Yn in Hn2}.
    
    For $n \geq 5$, observe that it suffices to prove that the inclusion map $$\ind \lno Z_n^{(3)} \setminus N_{Z_n^{(3)}} \lsq v_{2n}^2 \rsq \rno \hookrightarrow \ind \lno Z_n^{(3)} \setminus \lcu v_{2n}^2 \rcu \rno $$ is null-homotopic, since the result then follows from \Cref{lemma: finding homotopy using link and deletion ind complex}, and \Cref{lemma: link of v(2n)2 in Zn3}, \Cref{lemma: deletion of v(2n)2 in Zn3}.

    For this, consider the induced subgraph $Z_n^{(3)} \setminus \lcu v_{2n}^2, v_{2n+1}^2\rcu$ of $Z_n^{(3)}$ (see \Cref{fig: graph of Zn3 - {v-2n(2), v-(2n+1)(2)} in ind complex}).

    \begin{figure}[H]
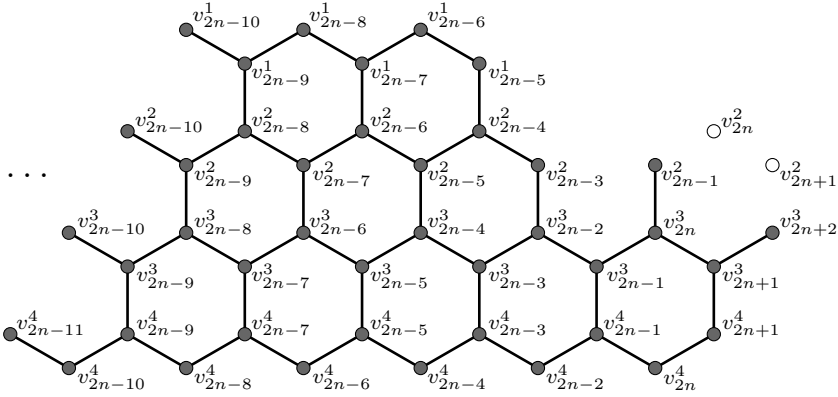

    \centering

        \caption{The graph $Z_n^{(3)} \setminus \lcu v_{2n}^2, v_{2n+1}^2\rcu$.}
        \label{fig: graph of Zn3 - {v-2n(2), v-(2n+1)(2)} in ind complex}
    \end{figure}
    
    We then have the following inclusion map, $$\ind \lno Z_n^{(3)} \setminus N_{Z_n^{(3)}} \lsq v_{2n}^2 \rsq \rno \hookrightarrow \ind \lno Z_n^{(3)} \setminus \lcu v_{2n}^2, v_{2n+1}^2\rcu \rno \hookrightarrow \ind \lno Z_n^{(3)}\setminus \lcu v_{2n}^2 \rcu \rno.$$

    In $Z_n^{(3)} \setminus \lcu v_{2n}^2, v_{2n+1}^2\rcu$, we fold $v_{2n+1}^3$ using $v_{2n-1}^2$. This procedure leaves $v_{2n+2}^3$ as an isolated vertex (see \Cref{fig: foldings in Zn3 - {v-2n(2), v-(2n+1)(2)}}). 

    \begin{figure}[H]
    \centering

        \caption{Foldings in the graph $Z_n^{(3)} \setminus \lcu v_{2n}^2, v_{2n+1}^2\rcu$.}
        \label{fig: foldings in Zn3 - {v-2n(2), v-(2n+1)(2)}}
    \end{figure}
    
    Thus, using the fold Lemma (\Cref{foldlemma}), we have
    \begin{align*}
       \ind \lno Z_n^{(3)} \setminus \lcu v_{2n}^2, v_{2n+1}^2\rcu \rno  &\simeq \ind \lno \lno Z_{n-3}^{(3)} \setminus \lcu v_{2n}^2, v_{2n+1}^2, v_{2n+1}^3, v_{2n+2}^3\rcu \rno \sqcup \lcu v_{2n+2}^3 \rcu \rno\\
        &\simeq \ind \lno Z_{n-3}^{(3)} \setminus \lcu v_{2n}^2, v_{2n+1}^2, v_{2n+1}^3, v_{2n+2}^3\rcu \rno *  \lcu v_{2n+2}^3\rcu \\
        &\simeq pt. 
    \end{align*}

    In other words, $\ind \lno Z_n^{(3)} \setminus \lcu v_{2n}^2, v_{2n+1}^2\rcu \rno  $ is contractible. Together with \Cref{prop: null-homotopy through factor}, this implies that the inclusion map,,
    $$\ind \lno Z_n^{(1)} \setminus N_{Z_n^{(1)}} \lsq v_{2n-1}^2 \rsq \rno \hookrightarrow \ind \lno Z_n^{(3)} \setminus \lcu v_{2n}^2, v_{2n+1}^2\rcu \rno,$$
    is null-homotopic. Which in turn implies that, $$\ind \lno Z_n^{(1)} \setminus N_{Z_n^{(1)}} \lsq v_{2n-1}^2 \rsq \rno \hookrightarrow \ind \lno Z_n^{(1)}\setminus \lcu v_{2n-1}^2 \rcu \rno,$$
    is null-homotopic. This completes the proof of \Cref{thm: Homotopy type of ind complex of Zn3}.
\end{proof}

\subsection{Independence complex of $H_n^3$}

We now discuss the homotopy type of the independence complex of $H_n^3$, denoted $\ind \lno H_n^3 \rno$. Recall that $H_n^3 = H_{1 \times 3 \times n}$. We do this by examining the link and deletion of $v_{2n}^2$ in $\ind \lno H_n^3 \rno$, that is, $\ind \lno H_n^3 \setminus N_{H_n^3} \lsq v_{2n}^2 \rsq \rno$ and $\ind \lno H_n^3 \setminus \lcu v_{2n}^2 \rcu \rno$, respectively, and then analyzing the inclusion maps between these two complexes.

\begin{lemma}\label{lemma: link of v(2n)2 in Hn3}
    For $n\geq 3$, $\ind \lno H_n^3 \setminus N_{H_{n}^3} \lsq v_{2n}^2 \rsq \rno \simeq \Sigma^4 \lno \ind \lno Z_{n-2}^{(1)} \rno \rno$.
\end{lemma}

\begin{figure}[H]
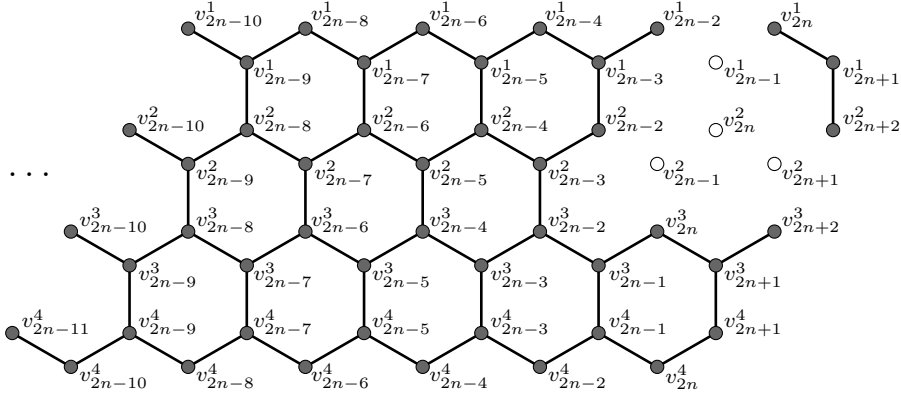

\centering

    \caption{The graph $H_n^3 \setminus N_{H_{n}^3} \lsq v_{2n}^2 \rsq$.}
    \label{fig: graph of Hn3 - N[v(2n)2] in ind complex}
\end{figure}

\begin{proof}
    Observe that in $ H_n^3 \setminus N_{H_{n}^3} \lsq v_{2n}^2 \rsq$ (see \Cref{fig: graph of Hn3 - N[v(2n)2] in ind complex}), we fold $v_{2n-4}^1$ and $v_{2n-2}^2$ using $v_{2n-2}^1$. On the other hand, we fold $v_{2n}^3$ and $v_{2n+1}^4$ using $v_{2n+2}^3$. Following the folding of $v_{2n+1}^4$, we fold $v_{2n-1}^3$ and $v_{2n-2}^4$ using $v_{2n}^4$. At last, we fold $v_{2n}^1$ using $v_{2n+2}^2$ (see \Cref{fig: foldings in Hn3 - N[v(2n)2]}).

    \begin{figure}[h!]
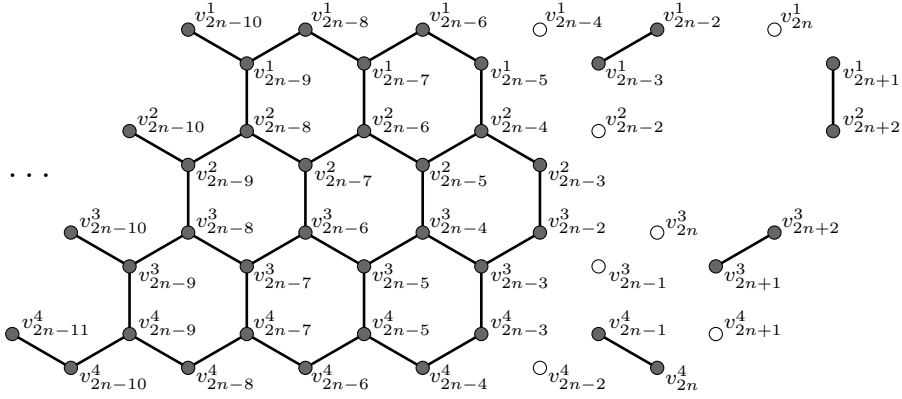

    \centering

        \caption{Foldings in the graph $H_n^3 \setminus N_{H_{n}^3} \lsq v_{2n}^2 \rsq$.}
        \label{fig: foldings in Hn3 - N[v(2n)2]}
    \end{figure}

    Therefore, using the fold Lemma (\Cref{foldlemma}), we have
    \begin{align*}
        \ind \lno H_n^3 \setminus N_{H_{n}^3} \lsq v_{2n}^2 \rsq \rno &\simeq \ind \lno Z_{n-2}^{(1)} \sqcup P_2 \sqcup P_2 \sqcup P_2 \sqcup P_2 \rno\\
        &\simeq \Sigma^4 \lno \ind \lno Z_{n-2}^{(1)} \rno \rno.
    \end{align*}

    This completes the proof \Cref{lemma: link of v(2n)2 in Hn3}.
\end{proof}

\begin{lemma}\label{lemma: deletion of v(2n)2 in Hn3}
    For $n\geq 4$, $$\ind \lno H_n^3 \setminus \lcu v_{2n}^2 \rcu \rno \simeq \Sigma^9 \lno \ind \lno Z_{n-3}^{(3)} \rno \rno \vee \Sigma^4 \lno \ind \lno Z_{n-2}^{(2)} \rno \rno.$$
\end{lemma}

\begin{figure}[h!]
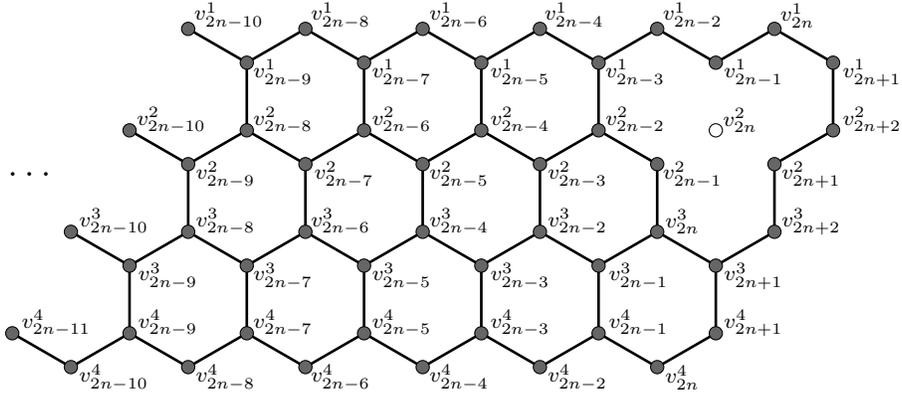

\centering

    \caption{The graph $U_n^{(1)} = H_n^3 \setminus \lcu v_{2n}^2 \rcu$.}
    \label{fig: graph of Un1 in ind complex}
\end{figure}

\begin{proof}
    Let $U_n^{(1)} = H_n^3 \setminus \lcu v_{2n}^2 \rcu$ (see \Cref{fig: graph of Un1 in ind complex}). We first discuss the homotopy type of

    \begin{enumerate}
        \item $\ind \lno U_n^{(1)} \setminus N_{U_n^{(1)}} \lsq v_{2n+1}^3 \rsq \rno$; and,
        \item $\ind \lno U_n^{(1)} \setminus \lcu v_{2n+1}^3 \rcu \rno$.
    \end{enumerate}

    Later, we show that the inclusion map $$\ind \lno U_n^{(1)} \setminus N_{U_n^{(1)}} \lsq v_{2n+1}^3 \rsq \rno \hookrightarrow \ind \lno U_n^{(1)} \setminus \lcu v_{2n+1}^3 \rcu \rno,$$ is null-homotopic.

    \caseheading{Homotopy type of $\bm{{\ind} \lno U_{n}^{(1)} \setminus N_{U_{n}^{(1)}} \lsq v_{2n+1}^{3} \rsq \rno}$:}

    \begin{figure}[H]
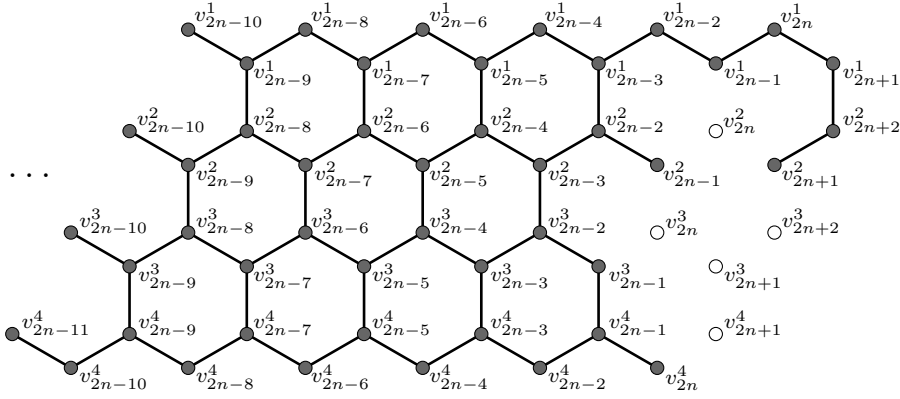

    \centering

        \caption{The graph $U_n^{(2)} = U_n^{(1)} \setminus N_{U_n^{(1)}} \lsq v_{2n+1}^3 \rsq$.}
        \label{fig: graph of Un2 in ind complex}
    \end{figure}

    Let $U_n^{(2)} = U_n^{(1)} \setminus N_{U_n^{(1)}} \lsq v_{2n+1}^3 \rsq$ (see \Cref{fig: graph of Un2 in ind complex}). In $U_n^{(2)}$, we fold $v_{2n+1}^1$ using $v_{2n+1}^2$ and following the folding of $v_{2n+1}^1$, we fold $v_{2n-2}^1$ using $v_{2n}^1$. On the other hand, we fold $v_{2n-1}^3$ and $v_{2n-2}^4$ using $v_{2n}^4$ (see \Cref{fig: foldings in Un2}).

    \begin{figure}[H]
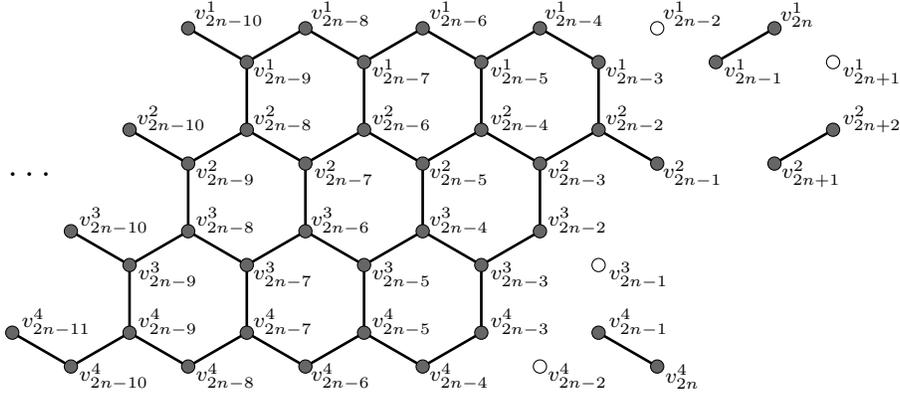

    \centering
        %
        \caption{Foldings in the graph $U_n^{(2)}$.}
        \label{fig: foldings in Un2}
    \end{figure}
    
    Therefore, using the fold Lemma (\Cref{foldlemma}), we have, $$\ind \lno U_n^{(2)} \rno \simeq \ind \lno Z_{n-2}^{(2)} \sqcup P_2 \sqcup P_2 \sqcup P_2 \rno \simeq \Sigma^3 \lno \ind \lno Z_{n-2}^{(2)} \rno \rno.$$

    Hence,
    \begin{align}\label{eqn: link of v(2n+1 - 3) in Un(1)}
        \ind \lno U_n^{(1)} \setminus N_{U_n^{(1)}} \lsq v_{2n+1}^3 \rsq \rno \simeq \Sigma^3 \lno \ind \lno Z_{n-2}^{(2)} \rno \rno.
    \end{align}

    \caseheading{Homotopy type of $\bm{{\ind} \lno U_{n}^{(1)} \setminus \lcu v_{2n+1}^{3} \rcu \rno}$:}

    \begin{figure}[H]
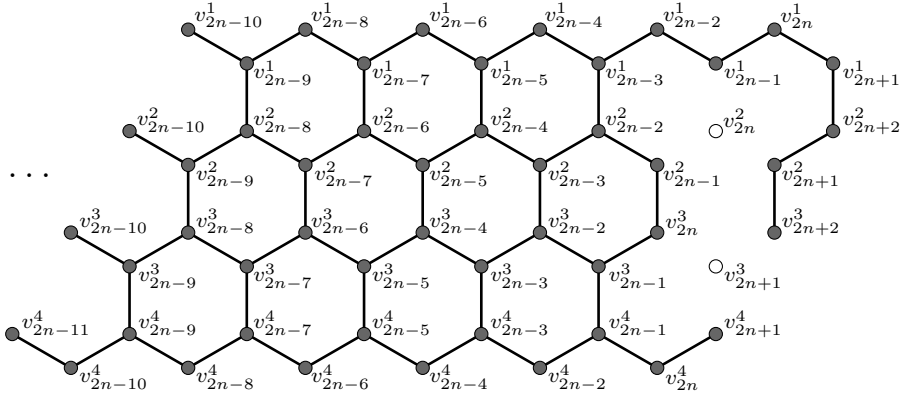

    \centering

        \caption{The graph $U_n^{(3)} = U_n^{(1)} \setminus \lcu v_{2n+1}^3 \rcu$.}
        \label{fig: graph of Un3 in ind complex}
    \end{figure}

    Let $U_n^{(3)} = U_n^{(1)} \setminus \lcu v_{2n+1}^3 \rcu$ (see \Cref{fig: graph of Un3 in ind complex}). We perform foldings on the vertices of $U_n^{(3)}$ according to the steps outlined below.

    \begin{enumerate}[label = (\roman*)]
        \item In $U_n^{(3)}$, we fold $v_{2n-1}^4$ using $v_{2n+1}^4$.
        \item Following the folding of $v_{2n-1}^4$, we fold $v_{2n-3}^3$ and $v_{2n-4}^4$ using $v_{2n-2}^4$.
        \item On the other hand, we fold $v_{2n+2}^2$ using $v_{2n+2}^3$.
        \item Once $v_{2n+2}^2$ is folded, we fold $v_{2n-1}^1$ using $v_{2n+1}^1$.
        \item After the folding of $v_{2n-1}^1$, we fold $v_{2n-4}^1$ and $v_{2n-2}^2$ using $v_{2n-2}^1$.
        \item Once $v_{2n-2}^2$ is folded, we fold $v_{2n-1}^3$ using $v_{2n-1}^2$.
        \item Observe that $v_{2n-3}^3$ was already folded (step (ii)), and with the folding of $v_{2n-1}^3$, we fold $v_{2n-4}^2$ using $v_{2n-2}^3$.
        \item Again, since $v_{2n-4}^1$ was already folded (step (v)), following the folding of $v_{2n-4}^2$, we fold $v_{2n-7}^1$ using $v_{2n-5}^1$.
        \item At last, after the folding of $v_{2n-7}^1$, we fold $v_{2n-10}^1$ and $v_{2n-8}^2$ using $v_{2n-8}^1$ (see \Cref{fig: foldings in Un3}).
    \end{enumerate}

    \begin{figure}[h!]
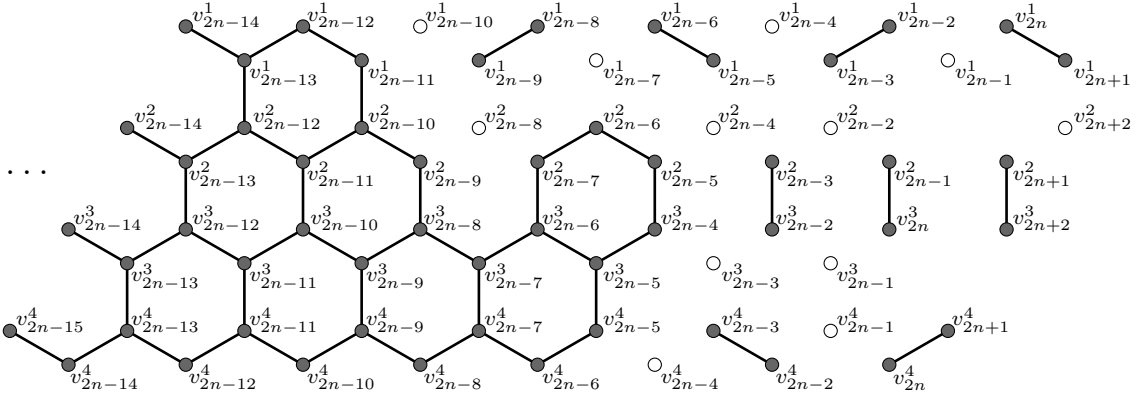

    \centering

        \caption{Folding in the graph $U_n^{(3)}$.}
        \label{fig: foldings in Un3}
    \end{figure}
    
    Thus, using the fold Lemma (\Cref{foldlemma}), we have
    \begin{align*}
        \ind \lno U_n^{(3)} \rno &\simeq \ind \lno Z_{n-3}^{(3)} \sqcup P_2 \sqcup P_2 \sqcup P_2 \sqcup P_2 \sqcup P_2 \sqcup P_2 \sqcup P_2 \sqcup P_2 \sqcup P_2 \rno\\
        &\simeq \Sigma^9 \lno \ind \lno Z_{n-3}^{(3)} \rno \rno.
    \end{align*}

    Hence,
    \begin{align}\label{eqn: deletion of v(2n+1 - 3) in Un(1)}
        \ind \lno U_n^{(1)} \setminus \lcu v_{2n+1}^3 \rcu \rno \simeq \Sigma^9 \lno \ind \lno Z_{n-3}^{(3)} \rno \rno.
    \end{align}

    \caseheading{The inclusion map $\bm{{\ind} \lno  U_{n}^{(1)} \setminus N_{U_{n}^{(1)}} \lsq v_{2n+1}^3 \rsq \rno \hookrightarrow {\ind} \lno  U_{n}^{(1)} \setminus \lcu v_{2n+1}^3 \rcu \rno}$ is null-homotopic:}

    \begin{figure}[H]
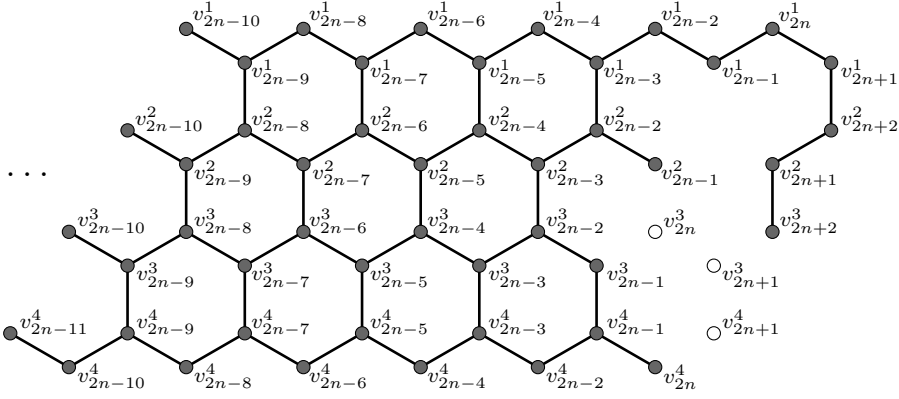

    \centering

        \caption{The graph $U_n^{(4)} = U_n^{(1)} \setminus \lcu v_{2n}^3, v_{2n+1}^3, v_{2n+1}^4 \rcu$.}
        \label{fig: graph of Un4 in ind complex}
    \end{figure}

    To prove this, we consider the induced subgraph $U_n^{(4)} = U_n^{(1)} \setminus \lcu v_{2n}^3, v_{2n+1}^3, v_{2n+1}^4 \rcu$, that is, $U_n^{(4)}$ is precisely $U_n^{(2)} \sqcup \lcu v_{2n+2}^3 \rcu$ (see \Cref{fig: graph of Un4 in ind complex}). We then have the following inclusion map, $$\ind \lno U_n^{(1)} \setminus N_{U_n^{(1)}} \lsq v_{2n+1}^3 \rsq \rno \hookrightarrow \ind \lno U_n^{(4)} \rno \hookrightarrow \ind \lno U_n^{(1)} \setminus \lcu v_{2n+1}^3 \rcu \rno.$$

    In $U_n^{(4)}$, we fold $v_{2n-3}^1$ and $v_{2n-3}^2$ using $v_{2n-1}^2$. On the other hand, we fold $v_{2n-1}^3$ and $v_{2n-2}^4$ using $v_{2n}^4$. After this procedure, we are left with two copies of an induced $P_2$ and one copy of an induced $P_7$ (see \Cref{fig: foldings in Un4}), along with $H_{n-2}^3 \setminus \lcu v_{2n-3}^1, v_{2n-3}^2, v_{2n-2}^2 \rcu$.
    
    \begin{figure}[h!]
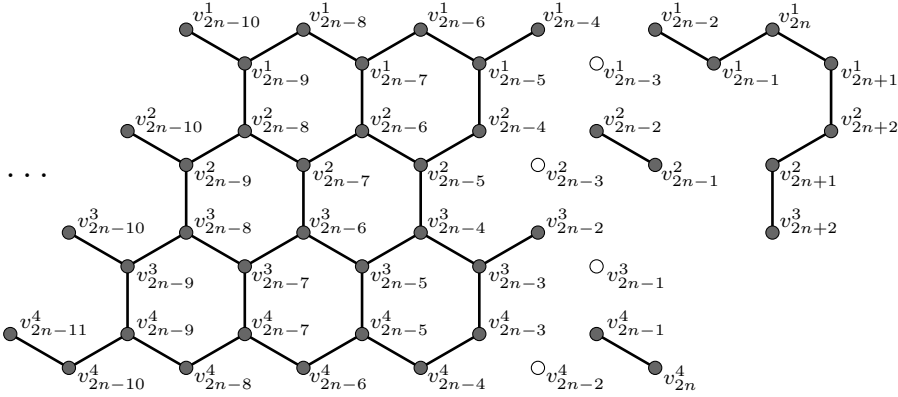

    \centering

        \caption{Foldings in the graph $U_n^{(4)}$.}
        \label{fig: foldings in Un4}
    \end{figure}
    
    Thus, using the fold Lemma (\Cref{foldlemma}), we have
    \begin{align*}
        \ind \lno U_n^{(4)} \rno &\simeq \ind \lno \lno H_{n-2}^3 \setminus \lcu v_{2n-3}^1, v_{2n-3}^2, v_{2n-2}^2 \rcu \rno \sqcup P_2 \sqcup P_2 \sqcup P_7 \rno\\
        &\simeq \Sigma^2 \lno \ind\lno H_{n-2}^3 \setminus \lcu v_{2n-3}^1, v_{2n-3}^2, v_{2n-2}^2 \rcu \rno \rno * \ind \lno P_7 \rno.
    \end{align*}
    Using \Cref{thm: homotopy type of ind complex of path graph}, we know that $\ind \lno P_7 \rno \simeq pt$. Hence, $$\ind \lno U_n^{(4)} \rno \simeq pt.$$

    In other words, $\ind \lno U_n^{(4)} \rno$ is contractible. This implies the inclusion map, $$\ind \lno U_n^{(1)} \setminus N_{U_n^{(1)}} \lsq v_{2n+1}^3 \rsq \rno \hookrightarrow \ind \lno U_n^{(4)} \rno,$$ is null-homotopic. Which in turn implies that the inclusion map $$\ind \lno U_n^{(1)} \setminus N_{U_n^{(1)}} \lsq v_{2n+1}^3 \rsq \rno \hookrightarrow \ind \lno U_n^{(1)} \setminus \lcu v_{2n+1}^3 \rcu \rno,$$ is null-homotopic. Therefore, using \Cref{lemma: finding homotopy using link and deletion ind complex}, \Cref{eqn: link of v(2n+1 - 3) in Un(1)} and \Cref{eqn: deletion of v(2n+1 - 3) in Un(1)}, we have

    \begin{align*}
        \ind \lno U_n^{(1)} \rno &\simeq \ind \lno U_n^{(1)} \setminus \lcu v_{2n+1}^3 \rcu \rno \vee \Sigma \lno \ind \lno U_n^{(1)} \setminus N_{U_n^{(1)}} \lsq v_{2n+1}^3 \rsq \rno \rno \\ 
        &\simeq \Sigma^9 \lno \ind \lno Z_{n-3}^{(3)} \rno \rno \vee \Sigma \lno \Sigma^3 \lno \ind \lno Z_{n-2}^{(2)} \rno \rno \rno\\
        &\simeq \Sigma^9 \lno \ind \lno Z_{n-3}^{(3)} \rno \rno \vee \Sigma^4 \lno \ind \lno Z_{n-2}^{(2)} \rno \rno.
    \end{align*}

    Hence, $$\ind \lno H_n^3 \setminus \lcu v_{2n}^2 \rcu \rno \simeq \Sigma^9 \lno \ind \lno Z_{n-3}^{(3)} \rno \rno \vee \Sigma^4 \lno \ind \lno Z_{n-2}^{(2)} \rno \rno.$$

    This completes the proof of \Cref{lemma: deletion of v(2n)2 in Hn3}.
\end{proof}

We now have all the necessary results to determine the homotopy type $\ind \lno H_{n}^{(3)} \rno$. We state this result in the following theorem.

\begin{theorem}\label{thm: Homotopy type of ind complex of Hn3}
    For $n\geq 1$, the independence complex of the triple hexagonal line tiling, that is, $\ind \lno H_n^3 \rno$, has the following homotopy type, $$\ind \lno H_n^3 \rno \simeq \begin{cases}
                \mathbb{S}^3 \vee \mathbb{S}^3, & \text{if } n=1;\\
                \mathbb{S}^4 \vee \mathbb{S}^4 \vee \mathbb{S}^4, & \text{if } n=2;\\
                \mathbb{S}^7, & \text{if } n=3;\\
                \mathbb{S}^6 \vee \mathbb{S}^{10}, \vee \mathbb{S}^{10} \vee \mathbb{S}^{10}, \vee \mathbb{S}^{10}, & \text{if } n=4;\\
                \Sigma^9 \lno \ind \lno Z_{n-3}^{(3)} \rno \rno \vee \Sigma^4 \lno \ind \lno Z_{n-2}^{(2)} \rno \rno \vee \Sigma^5 \lno \ind \lno Z_{n-2}^{(1)} \rno \rno, & \text{if } n \geq 5.
            \end{cases}$$
\end{theorem}

\begin{proof}
    Clearly the $H_1^3$ and $H_2^3$ are isomorphic to the graph $H_{1 \times 1 \times 3}$ (see \Cref{fig: graph of H1n in ind complex}) and $H_{1 \times 2  \times 3}$ (see \Cref{fig: graph of H(1_2_n) in ind complex}), respectively. Therefore, the result for the case $n=1$ and $2$ follows from \Cref{thm: homotopy type of H11n} and \Cref{thm: homotopy type of Hn2}. The proof for the remaining cases is omitted, as it follows the same pattern as the argument for $n \geq 5$. The intermediate graphs arising from the link and deletion operations are even simpler here, and their homotopy types can be computed directly via the fold lemma. Furthermore, the inclusion of the link into the deletion is null-homotopic, because the spheres appearing in the homotopy type of the link occur in strictly lower dimensions than those in the deletion.
    
    For the case when $n\geq 5$, observe that it suffices to prove that the inclusion map $$ \ind \lno H_n^3 \setminus N_{H_{n}^3} \lsq v_{2n}^2 \rsq \rno \hookrightarrow \ind \lno H_n^3 \setminus \lcu v_{2n}^2 \rcu \rno$$ is null-homotopic, since the result then follows from \Cref{lemma: link of v(2n)2 in Hn3}, \Cref{lemma: deletion of v(2n)2 in Hn3} and \Cref{lemma: finding homotopy using link and deletion ind complex}.

    For this, consider the induced subgraph $U_n^{(5)} = H_n^{(3)} \setminus \lcu v_{2n-1}^1, v_{2n}^2, v_{2n+1}^2\rcu$ of $H_n^{(3)}$ (see \Cref{fig: graph of Un5 in ind complex}). 

    \begin{figure}[H]
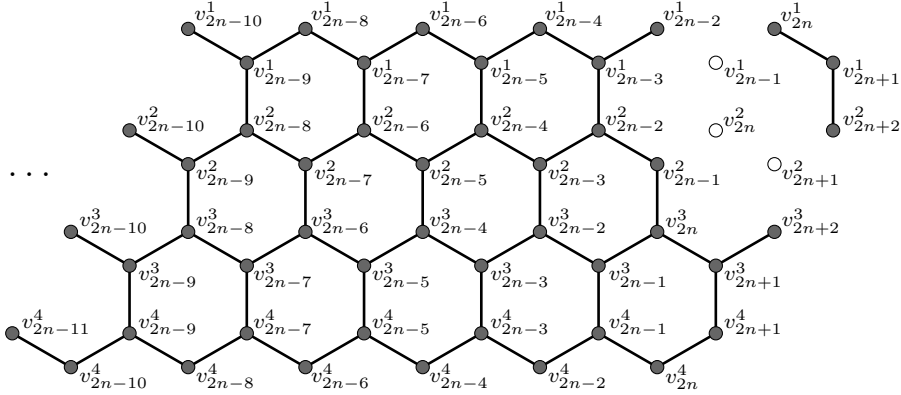

    \centering

        \caption{The graph $U_n^{(5)} = H_n^{(3)} \setminus \lcu v_{2n-1}^1, v_{2n}^2, v_{2n+1}^2\rcu$.}
        \label{fig: graph of Un5 in ind complex}
    \end{figure}
    
    We then have the following inclusion map, $$\ind \lno \ind \lno H_n^3 \setminus N_{H_{n}^3} \lsq v_{2n}^2 \rsq \rno \rno \hookrightarrow \ind \lno U_n^{(5)} \rno \hookrightarrow \ind \lno \ind \lno H_n^3 \setminus \lcu v_{2n}^2 \rcu \rno \rno.$$

    The same procedure used to determine the homotopy type of $\ind \lno H_n^3 \setminus N_{H_{n}^3} \lsq v_{2n}^2 \rsq \rno$ in the proof of \Cref{lemma: link of v(2n)2 in Hn3} applies here as well. The only exception is that here, this procedure leaves $v_{2n-1}^2$ as an isolated vertex (see \Cref{fig: foldings in Un5}). 

    \begin{figure}[h!]
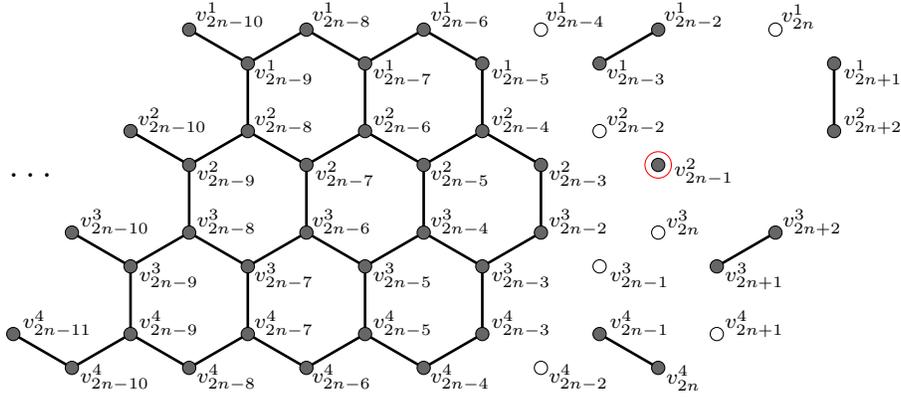

    \centering

        \caption{Foldings in the graph $U_n^{(5)}$.}
        \label{fig: foldings in Un5}
    \end{figure}
    
    Thus, using the fold Lemma (\Cref{foldlemma}), we have
    \begin{align*}
        \ind \lno U_n^{(5)} \rno &\simeq \ind \lno Z_{n-2}^{(1)} \sqcup P_2 \sqcup P_2 \sqcup P_2 \sqcup P_2 \sqcup \lcu v_{2n}^3 \rcu \rno \\
        & \simeq \Sigma^4 \lno \ind \lno Z_{n-2}^{(1)} \rno \rno * \ind \lno \lcu v_{2n-1}^2 \rcu \rno\\
        & \simeq pt.
    \end{align*}

    In other words, $\ind \lno U_n^{(5)} \rno  $ is contractible. This implies the inclusion map,
    $$\ind \lno H_n^3 \setminus N_{H_{n}^3} \lsq v_{2n}^2 \rsq \rno \hookrightarrow \ind \lno U_n^{(5)} \rno,$$
    is null-homotopic. Which in turn implies that, $$\ind \lno H_n^3 \setminus N_{H_{n}^3} \lsq v_{2n}^2 \rsq \rno \hookrightarrow \ind \lno H_n^3 \setminus \lcu v_{2n}^2 \rcu \rno,$$
    is null-homotopic. This completes the proof of \Cref{thm: Homotopy type of ind complex of Hn3}.
\end{proof}

\begin{figure}[h!]
  \centering
  \captionsetup{justification=centering}
  \includegraphics[scale = 0.9]{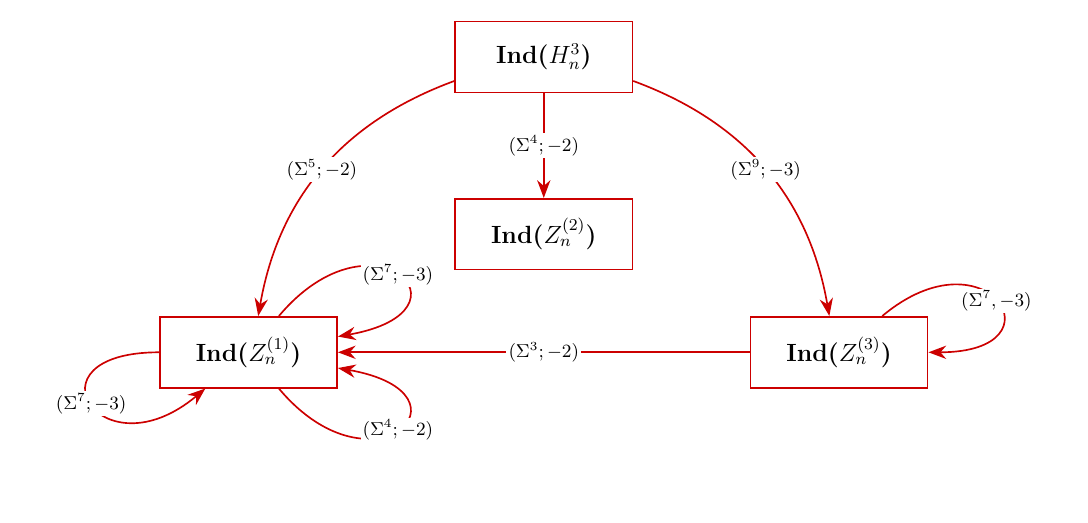}
  \caption{Flow chart illustrating the reduction steps and homotopy equivalences used in computing the homotopy type of $H_n^{3}$.}
  \label{fig: flowchart}
\end{figure}

The flowchart presented in \Cref{fig: flowchart} summarizes the results obtained and the procedure for computing the homotopy type of $H_n^{3}$. Each arrow emanating from a block indicates that the independence complex in that block is homotopy equivalent to the independence complex in the block to which the arrow points. An arrow represents a wedge–sum decomposition, and its label $(\Sigma^{r}; s)$ specifies the suspension shift $\Sigma^{r}$ that appears in the homotopy type, while the integer $s$ records the amount by which the variable $n$ changes in the corresponding step.

\section*{Acknowledgements}
Himanshu Chandrakar gratefully acknowledges the assistance provided by the Council of Scientific and Industrial Research (CSIR), India, through grant 09/1237(15675)/2022-EMR-I. 

\nocite{}
\printbibliography 
    
\end{document}